\documentclass[12pt]{article}
\usepackage[margin=1in]{geometry}
\usepackage{hyperref}
\usepackage{amsfonts}
\usepackage{mathrsfs}
\usepackage{comment}
\usepackage{amsmath}
\usepackage{amssymb}
\usepackage{amsthm}
\usepackage{amscd}
\usepackage{graphicx}
\usepackage{indentfirst}
\usepackage[all]{xy}
\usepackage{titlesec}
\usepackage{enumerate}
\usepackage{bm}
\usepackage{enumitem}
\usepackage{changepage}
\newtheorem{theorem}{Theorem}[section]
\newtheorem{lemma}[theorem]{Lemma}
\newtheorem{proposition}[theorem]{Proposition}

\newtheorem{conjecture}[theorem]{Conjecture}

\newtheorem{remark}[theorem]{Remark}
\newtheorem{claim}[theorem]{Claim}

\usepackage[utf8]{inputenc}

\newcommand{\ma}{\mathcal}
\newcommand{\mr}{\mathscr}
\newcommand{\s}{\subseteq}

\newcommand{\fr}{\frac}
\newcommand{\lc}{\lceil}
\newcommand{\rc}{\rceil}
\newcommand{\lf}{\lfloor}
\newcommand{\rf}{\rfloor}

\begin{document}

\title{Sparse hypergraphs: new bounds and constructions}
\date{}
\author{Gennian Ge\footnote{School of Mathematical Sciences, Capital Normal University, Beijing 100048, China. Email: gnge@zju.edu.cn.}~
and
Chong Shangguan\footnote{Research Center for Mathematics and Interdisciplinary Sciences, Shandong University, Qingdao 266237, China. Email: theoreming@163.com}
}

\maketitle

\begin{abstract}
    Let $f_r(n,v,e)$ denote the maximum number of edges in an $r$-uniform hypergraph on $n$ vertices, in which the union of any $e$ distinct  edges contains at least $v+1$ vertices. The study of $f_r(n,v,e)$ was initiated by Brown, Erd{\H{o}}s and S{\'o}s more than forty years ago.
    In the literature, the following conjecture is well known.

    \vspace{5pt}

    \textbf{Conjecture:} $n^{k-o(1)}<f_r(n,er-(e-1)k+1,e)=o(n^k)$ holds for all fixed integers $r>k\ge 2$ and $e\ge 3$ as $n\rightarrow\infty$.
    \vspace{5pt}

    For $r=3, e=3, k=2$, the bound $n^{2-o(1)}<f_3(n,6,3)=o(n^2)$ was proved by the celebrated (6,3)-theorem of Ruzsa and Szemer{\'e}di. In this paper, we add more evidence for the validity of the conjecture. On one hand, using the hypergraph removal lemma we show that the upper bound part of the conjecture is true for all fixed integers $r\ge k+1\ge e\ge3$.
    On the other hand, using tools from additive number theory we present several constructions showing that the lower bound part of the conjecture is true for $r\ge3$, $k=2$ and $e=4,5,7,8$.
    Prior to our results, all known constructions that match the conjectured lower bound satisfy either $r=3$ or $e=3$. Our constructions are the first ones in the literature that break this barrier.
\end{abstract}


\vspace{5pt}



\section{Introduction}
\noindent Since the pioneering work of Tur\'an \cite{turan}, the study of Tur\'an-type problems has been playing a central role in the field of extremal graph theory.
 In this paper we consider a classical hypergraph Tur\'an problem introduced by Brown, Erd{\H{o}}s and S{\'o}s \cite{bes2} in the early 1970's.

  Let us begin with some basic notations.
  Given a finite set $V(\ma{H})$, a {\it hypergraph} $\ma{H}$ is simply a family of subsets of $V(\ma{H})$, where the members of $V(\ma{H})$ and $\ma{H}$ are called {\it vertices} and {\it edges}, respectively.
  $\ma{H}$ is said to be an {\it $r$-uniform hypergraph} (or $r$-graph for short) if $|A|=r$ for each $A\in \ma{H}$.
  Moreover, $\ma{H}$ is called {\it linear} if for any distinct $A,B\in \ma{H}$, $|A \cap B|\le1$.
  For a set $\mr{H}$ of $r$-graphs, an $r$-graph is said to be {\it $\mr{H}$-free} if it contains none of the members of $\mr{H}$ as a subhypergraph. The {\it Tur{\'a}n number $ex_r(n,\mr{H})$} is the maximum number of edges in an $\mr{H}$-free $r$-graph on $n$ vertices. For simplicity we view the vertex set as the set (or a subset) of the first $n$ integers, denoted as $[n]:=\{1,\ldots,n\}$. 

  For integers $r\ge 2,e\ge 2,v\ge r+1$, let $\ma{G}_r(v,e)$ be the family of all $r$-graphs formed by $e$ edges and at most $v$ vertices, that is,
$$\ma{G}_r(v,e)=\big\{\ma{H}\s\binom{[n]}{r}:|\ma{H}|=e,|V(\ma{H})|\le v\big\},$$
where for any finite set $X\s[n]$, $\binom{X}{r}$ denotes the family of $\binom{|X|}{r}$ distinct $r$-subsets of $X$.
It is clear by definition that an $r$-graph is $\ma{G}_r(v,e)$-free if and only if the union of any $e$ distinct edges of it contains at least $v+1$ vertices.
Such $r$-graphs are also termed {\it sparse} \cite{Furediconst}.
As in the previous papers (see, e.g. \cite{t=3}), we use the notation $f_r(n,v,e):=ex(n,\ma{G}_r(v,e))$.



In the literature, the study of $f_r(n,v,e)$ for $e=2$ or $r=2$ has been quite extensive (see, e.g. \cite{Erdos1964,Erdosr=2,Rodlpacking}). In this paper we are interested in the behavior of $f_r(n,v,e)$ when $r,v,e$ are fixed integers satisfying $r\ge 3,e\ge 3, v\ge r+1$ and $n$ tends to infinity. It was shown in \cite{bes2} that in general
\begin{equation}\label{BESbound}
  \begin{aligned}
   \Omega(n^{\fr{er-v}{e-1}})=f_r(n,v,e)=O(n^{\lc\fr{er-v}{e-1}\rc}),
  \end{aligned}
\end{equation}
 where the upper bound follows from a double counting argument and the lower bound was proved by a standard probabilistic method.
By \eqref{BESbound} it is clear that for fixed integers $r>k\ge 2$ and $e\ge3$, $$f_r(n,er-(e-1)k,e)=\Theta(n^k).$$
However, the asymptotic behaviour of $f_r(n,er-(e-1)k+1,e)$ is much more difficult to determine. 
Indeed, there is a well known conjecture (see \cite{bes2,t=3}) which states the following:

\begin{conjecture}\label{conjecture}
    $n^{k-o(1)}<f_r(n,er-(e-1)k+1,e)=o(n^k)$ holds for all fixed integers $r>k\ge 2$ and $e\ge 3$. 
\end{conjecture}

The initial case $r=3,k=2,e=3$ was famously solved by the (6,3)-theorem of Ruzsa and Szemer{\'e}di \cite{RS}, which stated that

\begin{equation}\label{63}
    \begin{aligned}
        n^{2-o(1)}<f_3(n,6,3)=o(n^2).
    \end{aligned}
\end{equation}

\noindent It is noteworthy that \eqref{63} is truly a cornerstone in the field of extremal graph theory. Indeed, the upper bound was proved by the triangle removal lemma, which was a remarkable application of the regularity lemma \cite{Szemeredi}. 
On the other hand, the proof of the lower bound was based on Behrend's construction \cite{Behrend} on 3-term arithmetic progression free sets, and it essentially related extremal graph theory to additive combinatorics, which turned out to be very influential.

The result of Ruzsa and Szemer{\'e}di was extended by Erd{\H{o}}s, Frankl and R{\"o}dl \cite{erdos2} to

\begin{equation}\label{363}
    \begin{aligned}
        n^{2-o(1)}<f_r(n,3r-3,3)=o(n^2)
    \end{aligned}
\end{equation}

\noindent for arbitrary fixed $r\ge3$, and further by Alon and Shapira \cite{t=3} to

\begin{equation}\label{3r63}
    \begin{aligned}
        n^{k-o(1)}<f_r(n,3r-2k+1,3)=o(n^k)
    \end{aligned}
\end{equation}

  \noindent for arbitrary fixed $r>k\ge2$. In \cite{t=3} the authors also showed that $n^{2-o(1)}<f_3(n,7,4)$ and $n^{2-o(1)}<f_3(n,8,5)$. Prior to this paper, these two sporadic cases together with the lower bound of (\ref{3r63}) (note that (\ref{63}) and (\ref{363}) are in fact special cases of (\ref{3r63})) were all known constructions that achieve the lower bound of Conjecture \ref{conjecture}. $\rm{S\acute{a}rk\ddot{o}zy}$ and Selkow \cite{other2,other1} considered the upper bound of the conjecture and they showed that

\begin{equation}\label{4r63}
        \begin{aligned}
            f_r(n,4r-3k+1,4)=o(n^k)
        \end{aligned}
    \end{equation}

\noindent for all fixed $r>k\ge 3$, and

\begin{equation}\label{asymtotic}
        \begin{aligned}
            f_r(n,er-(e-1)k+\lf\log_2 e\rf,e)=o(n^k)
        \end{aligned}
    \end{equation}

\noindent for all fixed $r>k\ge 2$ and $e\ge 3$. It is clear that (\ref{asymtotic}) implies the upper bounds stated in (\ref{63}), (\ref{363}) and (\ref{3r63}). However, there is still a gap from the conjectured upper bound for $e\ge 4$. Later, Nagle, R\"odl and Schacht \cite{nagle-rodl-schacht} showed that

\begin{equation}\label{added}
    f_r(n,(k+1)r-k^2+1,k+1)=o(n^k)
\end{equation}

\noindent for arbitrary fixed $r\ge k+1=e\ge 3$. A recent paper of Solymosi and Solymosi \cite{smallcores} proved that $f_3(n,14,10)=o(n^2)$, which improves the result of (\ref{asymtotic}) for the special case $r=3$, $k=2$ and $e=10$. This is the first improvement of the asymptotic bound (\ref{asymtotic}) in the silence over ten years. Finally, we remark that by \eqref{BESbound} it is easy to see that
\begin{equation}\label{randombound}
  \Omega(n^{k-\frac{1}{e-1}})=f_r(n,er-(e-1)k+1)=O(n^{k}).
\end{equation}
Recently, the above lower bound was slightly improved by Shangguan and Tamo \cite{shangguan2019universally} to
$$f_r(n,er-(e-1)k+1)=\Omega\big(n^{k-\frac{1}{e-1}}(\log n)^{\frac{1}{e-1}}\big).$$

In this paper, we add more evidence for the validity of Conjecture \ref{conjecture}. We will summarize our main results in the remaining part of this section. We will frequently use the standard Bachmann-Landau notations $\Omega(\cdot),\Theta(\cdot),O(\cdot)$ and $o(\cdot)$, whenever the constants are not important. All logarithms are of base 2.

\subsection{Upper bounds from the hypergraph removal lemma}



\noindent As mentioned previously, (\ref{63}) was proved by the triangle removal lemma, which is the first version of the removal lemma. Indeed, the upper bounds in (\ref{363}), (\ref{4r63}) and \eqref{added} were all proved by different versions of removal lemmas, namely, the graph removal lemma for complete graphs \cite{erdos2}, the hypergraph removal lemma for $K_4^3$ \cite{franklhypergraphremoval} and the hypergraph removal lemma for $K_{k+1}^k$ \cite{gowers2}, \cite{rodl2}, \cite{rodl3}, where $K_t^r$ denotes the complete $r$-graph on $t$ vertices.

After the efforts of many researchers (see, e.g. \cite{gowers1}, \cite{gowers2}, \cite{rodl1}, \cite{rodl2}, \cite{rodl3}, \cite{rodl4}), the following modern version of the hypergraph removal lemma is eventually known.

\begin{lemma}[Hypergraph removal lemma, see, e.g. \cite{graphremovallemmas} Theorem 1.2]\label{hypergraph}
    For any $r$-graph $G$ and any $\epsilon>0$, there exists $\delta>0$ such that any $r$-graph on $n$ vertices with at most $\delta n^{v(G)}$ copies of $G$ can be made $G$-free by removing at most $\epsilon n^r$ edges.
\end{lemma}

Note that here $v(G)=|V(G)|$ is the number of vertices of $G$. Using Lemma \ref{hypergraph}, we are able to prove the following theorem.

\begin{theorem}\label{upperbd}
   $f_r(n,er-(e-1)k+1,e)=o(n^k)$ holds for all fixed integers 
   $r\ge k+1\ge e\ge 3$.
\end{theorem}

This theorem implies that the upper bound part of Conjecture \ref{conjecture} is true for all fixed integers $r\ge k+1\ge e\ge3$. 
The first unsettled case is to determine whether $f_3(n,7,4)=o(n^2)$, which is known as the (7,4)-problem.


Our proof of Theorem \ref{upperbd} implies the following new asymptotic upper bound. 

\begin{theorem}\label{generalupperbd}
    For fixed integers $r,k,e,i$ satisfying $i\ge 0,r\ge k+i+1\ge2$ and $\binom{k+i+1}{k}\ge e\ge 3$, it holds that $f_r(n,er-(e-1)k+i+1,e)=o(n^k)$.
\end{theorem}

It is easy to see that our result improves (\ref{asymtotic}) when $r,e,k,i$ satisfy $r\ge k+i+1\ge2$, $\binom{k+i+1}{k}\ge e$ and $\lf\log_2 e\rf\ge i+2$.
For example, taking $i=1$, Theorem \ref{generalupperbd} gives $f_r(n,10r-25,e)=o(n^3)$ for $r\ge 5$, while (\ref{asymtotic}) gives $f_r(n,10r-24,e)=o(n^3)$ (note that the conjectured upper bound is $f_r(n,10r-26,e)=o(n^3))$.


\begin{remark}
After submitting the paper, we learned from a reviewer that Theorem \ref{upperbd} can also be proved by a reduction to \eqref{added}. However, this connection was not explicitly stated in \cite{nagle-rodl-schacht}. Compared with \cite{nagle-rodl-schacht}, our main contribution is that we apply the full power of the hypergraph removal lemma to obtain Theorems \ref{upperbd} and \ref{generalupperbd} (we apply the hypergraph removal lemma for $K^k_{r}$, while \cite{nagle-rodl-schacht} applied the lemma for $K^k_{k+1}$). As far as we know, the latter does not follow directly from \eqref{added}.
\end{remark}

\subsection{Lower bounds from rainbow-cycle-free hypergraphs and additive number theory}

\noindent In the literature, sparse hypergraphs attaining the lower bound of Conjecture \ref{conjecture} are quite rare. Indeed, prior to our results, there was no construction that matches the conjectured lower bound when $\min \{e,r\}\ge 4$.

We break this barrier by showing that the lower bound of Conjecture \ref{conjecture} is true for $r\ge 3$, $k=2$ and $e=4,5,7,8$. A novel idea of our approach is that we find that rainbow-cycle-free linear hypergraphs (whose definition is given in Section \ref{3}) are good candidates for constructing sparse hypergraphs. Rainbow cycles are defined on linear $r$-partite $r$-graphs. Loosely speaking, they are Berge cycles \cite{hypercycle2,hypercycle1} with an additional property: the vertices of the cycle are located in distinct vertex parts of the $r$-partite hypergraph. 

We prove the following theorem.

\begin{theorem}\label{main}
  Let $r\ge 3$ be a fixed integer and $\ma{H}$ be a linear $r$-partite $r$-graph. Assume that $\ma{H}$ contains no rainbow cycles of length three or four. Then 
  $\ma{H}$ is simultaneously $\ma{G}_r(3r-3,3)$-free, $\ma{G}_r(4r-5,4)$-free, $\ma{G}_r(5r-7,5)$-free, $\ma{G}_r(7r-11,7)$-free and $\ma{G}_r(8r-13,8)$-free.
\end{theorem}

The proofs for $e=3,4,5$ are relatively simple, while the proofs for $e=7,8$ are much more involved. The key ingredient (in the proofs for $e=7,8$) is that under the assumption of Theorem \ref{main}, if $\ma{H}$ is not
$\ma{G}_r(6r-9,6)$-free, then any six edges $A_1,\ldots,A_6$ of $\ma{H}$ with $|\cup_{i=1}^6 A_i|\le 6r-9$ must satisfy $|\cup_{i=1}^6 A_i|=6r-9$ and crucially, up to isomorphism they have only one possible configuration (see Theorem \ref{clafy6r-9} below).

According to Theorem \ref{main}, to construct sparse hypergraphs achieving the lower bound of the conjecture, it suffices to construct sufficiently large rainbow-cycle-free hypergraphs. Additive number theory is a useful tool for constructing hypergraphs with certain forbidden subhypergraphs, see, e.g. \cite{ipp,Furediconst,RS,shf}. The idea is to characterize the forbidden subhypergraphs by a couple of equations. Then, the existence of an appropriately defined solution-free set (a {\it solution-free set} is a set which contains no nontrivial solution to certain prescribed equations) will guarantee the existence of the desired hypergraph which contains no forbidden subhypergraphs.

Using this approach, we prove the following result.

\begin{theorem} \label{rainbow}
    For any fixed integer $r\ge3$ and sufficiently large $n$, 
    there exists a linear $r$-partite $r$-graph $\ma{H}$ on $n$ vertices with $|\ma{H}|>n^{2-o(1)}$, which contains no rainbow cycles of length three or four.
\end{theorem}

The following result is a direct consequence of Theorems \ref{main} and \ref{rainbow}.

\begin{proposition}\label{4578}
$f_r(n,er-2e+3,e)>n^{2-o(1)}$ holds for all fixed integers $r\ge 3$ and $e=4,5,7,8$.
\end{proposition}


\subsection{Linear 3-graphs containing neither grids nor triangles}

\noindent In this subsection we introduce an extremal hypergraph problem of F\"uredi and Ruszink\'o \cite{Furediconst} which is closely related to the sparse hypergraphs.

An $r$-graph is called an {\it $r\times r$ grid} if it is isomorphic to a pattern of $r$ horizontal and $r$ vertical lines, that is, a family of $2r$ distinct edges $\{A_1,\ldots,A_r,B_1,\ldots,B_r\}$ such that $A_i\cap A_j=B_i\cap B_j=\emptyset$ for all $1\le i<j\le r$ and $|A_i\cap B_j|=1$ for all $1\le i,j\le r$. A family of three distinct edges $C_1,C_2,C_3$ form a {\it triangle} if $|C_1\cap C_2|=|C_1\cap C_3|=|C_2\cap C_3|=1$, and the three singletons are pairwise distinct. Let $G_{r\times r}$ and $T_3$ denote $r\times r$ grids and triangles, respectively. Let $I_{\ge 2}$ denote the class of $r$-graphs composed of two edges intersecting in at least two common vertices. Clearly an $r$-graph is linear if only if it is $I_{\ge 2}$-free.

F\"uredi and Ruszink\'o \cite{Furediconst} considered the Tur\'an number of linear $r$-graphs which contain neither grids nor triangles. They showed that for any fixed integer $r\ge 4$,
\begin{equation*}
    n^{2-o(1)}<ex_r(n,\{I_{\ge 2},T_3,G_{r\times r}\})=o(n^2).
\end{equation*}

\noindent However, for $r=3$ they can only prove
\begin{equation}\label{r=3}
    n^{1.6-o(1)}<ex_3(n,\{I_{\ge 2},T_3,G_{3\times 3}\})=o(n^2).
\end{equation}
The following two facts are easy to verify by definition: $(a)$ a linear 3-graph is $\ma{G}_3(6,3)$-free if and only if it is $T_3$-free; $(b)$ if a 3-graph is $\ma{G}_3(9,6)$-free, then it must be $G_{3\times 3}$-free.

We show that solution-free sets defined on a carefully chosen finite vector space
can be used to construct linear 3-graphs which are simultaneously $\ma{G}_3(6,3)$-free and $\ma{G}_3(9,6)$-free, thereby giving a new lower bound for \eqref{r=3}. More precisely, let $\mathbb{F}_7$ be the finite field of seven elements and $r(\mathbb{F}_7^n)$ be the maximum size of  a subset of $\mathbb{F}_7^n$ that contains no three
distinct elements $m_1,m_2,m_3$ with $m_1+2m_2=3m_3$.
We prove the following theorem.

 \begin{theorem}\label{96free}
There exists a linear 3-partite 3-graph with $3\cdot 7^{n}$ vertices and $r(\mathbb{F}_7^n)\cdot 7^n$ edges which is simultaneously $\ma{G}_3(6,3)$-free and $\ma{G}_3(9,6)$-free, thereby implying that
   $ex_3(3\cdot 7^{n},\{I_{\ge 2},T_3,G_{3\times 3}\})\ge r(\mathbb{F}_7^n)\cdot 7^n$.
 \end{theorem}

Noting that Lin and Wolf \cite{LinWolf} (see Theorem 1 of \cite{LinWolf}) proved that 
$r(\mathbb{F}_7^n)=\Omega(7^{2n/3})$, we immediately have the following proposition.

\begin{proposition}
$ex_3(n,\{I_{\ge 2},T_3,G_{3\times 3}\})=\Omega(n^{1.6})$ as $n\rightarrow\infty$.
\end{proposition}

It is clear that our new lower bound slightly improves that of \eqref{r=3}. However, it is plausible that one could have $ex_3(n,\{I_{\ge 2},T_3,G_{3\times3}\})>n^{2-o(1)}$.
Unfortunately, Theorem \ref{96free} can never be used to prove that lower bound, since a recent breakthrough of Ellenberg and Gijswijt \cite{ellenberg2016large} showed that $r(\mathbb{F}_7^n)<c^n$ for some explicit constant $c<7$.

Lastly, it follows from Theorem \ref{96free} and the lower bound on $r(\mathbb{F}_7^n)$ that $f_3(n,9,6)>n^{1.6}$. However, this bound is inferior to the random construction \eqref{randombound} giving that $f_3(n,9,6)=\Omega(n^{1.8})$. 
We remark that one could show $f_3(n,9,6)>n^{2-o(1)}$ if a question of Ruzsa \cite{ruzsa} on the size of certain solution-free set (see Section \ref{new-section-1} below for the details) can be answered affirmatively.


\subsection{A general upper bound for $f_r(n,v,e)$}
\noindent In this subsection we introduce a general upper bound for $f_r(n,v,e)$, as stated below. 

\begin{theorem}\label{generalF}
    For fixed integers $r\ge 2,e\ge 2,v\ge r+1$, write $er-v=p(e-1)+q$, where $1\le q\le e-1$. Then it holds that
    $$f_r(n,v,e)\le q\binom{n}{p+1}/\binom{r}{p+1}+(e-1-q)\binom{n}{p}/\binom{r}{p}.$$
    \noindent 
    In other words,
    $$f_r(n,v,e)\le q\binom{n}{\lc\fr{er-v}{e-1}\rc}/\binom{r}{\lc\fr{er-v}{e-1}\rc}+(e-1-q)\binom{n}{\lf\fr{er-v}{e-1}\rf}/\binom{r}{\lf\fr{er-v}{e-1}\rf}.$$
\end{theorem}

To the best of our knowledge, the theorem above provides the first general upper bound for $f_r(n,v,e)$. 
Unfortunately, it contributes nothing to the upper bound of  Conjecture \ref{conjecture}. 
Indeed, letting $v=er-(e-1)k+1$, the theorem above only gives that $$f_r(n,er-(e-1)k+1,e)\le(e-2)\binom{n}{k}/\binom{r}{k}+\binom{n}{k-1}/\binom{r}{k-1}.$$
However, Theorem \ref{generalF} may be useful for the study of other combinatorial objects, for example, combinatorial batch codes \cite{cbc}, 
perfect hash families \cite{shf}, 
and a related hypergraph extremal problem \cite{Glock2019,shangguan2020density}.


\subsection{Outline of the paper}
 \noindent The rest of this paper is organized as follows. In Section \ref{2} we will use the hypergraph removal lemma to prove Theorems \ref{upperbd} and \ref{generalupperbd}. The next three sections will be devoted to the constructions of sparse hypergraphs. In Section \ref{3}, we will introduce two important notions that are used in our constructions, namely, rainbow cycles and $R_L$-solution-free sets. We will also construct a sufficiently large $R_4$-solution-free set, 
 which is used in Section \ref{4} to construct rainbow-cycle-free hypergraphs. Then, in Section \ref{5} we will use rainbow-cycle-free hypergraphs to construct sparse hypergraphs that attain the lower bound of Conjecture \ref{conjecture} for fixed integers $r\ge 3$, $k=2$, and $e=4,5,7,8$.
 Theorems \ref{96free} and \ref{generalF} will be proved in Sections \ref{new-section-1} and \ref{new-section-2}, respectively.

\section{Sparse hypergraphs and the hypergraph removal lemma}\label{2}

\noindent The main task of this section is to prove that the right hand side of Conjecture 1 holds for all fixed integers $r,k,e$ satisfying $r\ge k+1\ge e\ge 3$. 

Using the hypergraph removal lemma described in Lemma \ref{hypergraph}, it is straightforward to deduce the following fact: for any given constant $\epsilon>0$, there exists some constant $\delta(\epsilon)>0$ so that if one must delete at least $\epsilon n^r$ edges to make an $r$-graph $\ma{H}$ on $n$ vertices $G$-free, then $\ma{H}$ must contain at least $\delta(\epsilon)n^{v(G)}$ copies of $G$.

\begin{proof}[\textbf{Proof of Theorem \ref{upperbd}}]
    Let $r,k,e$ be integers satisfying the assumption of the theorem and $\ma{H}$ be an $n$-vertex $r$-graph which is $\ma{G}_r(er-(e-1)k+1,e)$-free. Assume towards contradiction that $|\ma{H}|\ge\epsilon n^k$ holds for some constant $\epsilon>0$.

    We claim that for any $A_0\in\ma{H}$, there exist at most $e-2$ edges in $\ma{H}\setminus\{A_0\}$ such that each of them intersects $A_0$ in at least $k$ vertices. Indeed, if there exist $e-1$ distinct edges $A_1,\ldots,A_{e-1}$ with $|A_0\cap A_i|\ge k$ for every $1\le i\le e-1$, then 
    $$|\cup_{i=0}^{e-1}A_i|\le er-(e-1)k<er-(e-1)k+1,$$ which contradicts the fact that $\ma{H}$ is $\ma{G}_r(er-(e-1)k+1,e)$-free. Therefore, for any $A_0\in\ma{H}$, by removing at most $e-2$ edges which intersect $A_0$ in at least $k$ vertices, we can greedily construct a subhypergraph $\ma{H}'\s\ma{H}$ with $|\ma{H}'|\ge\fr{\epsilon}{e-1}n^k$ and $|A\cap B|\le k-1$ for all distinct $A,B\in\ma{H}'$. 

    To prove the theorem, we will make use of an auxiliary $k$-graph $\ma{H}^*$, as defined next.
    For each $r$-edge $A\in\ma{H}'$, we construct a hypergraph $K^k_r(A)$, which is the complete $k$-graph on the vertex set of $A$. In other words, $K^k_r(A)$ is formed by taking all of the $k$-element subsets of $A$. It is easy to see that there is a one-to-one correspondence between $A\in\ma{H}'$ and $K^k_r(A)\s\ma{H}^*$. The edges of $\ma{H}^*$ is formed by the union of the edges of all $K^k_r(A)$, that is, $$\ma{H}^*=\cup_{A\in\ma{H}'} K^k_r(A).$$
    As $|A\cap B|\le k-1$ for distinct $A,B\in\ma{H}'$, it follows that the $k$-graphs $K^k_r(A)$ and $K^k_r(B)$ are edge disjoint.

    To sum up, we construct a $k$-graph $\ma{H}^*$ which contains at least $|\ma{H}'|\ge\fr{\epsilon}{e-1}n^k$ edge disjoint copies of $K^k_r$. So one needs to delete at least $\fr{\epsilon}{e-1}n^k$ $k$-edges of $\ma{H}^*$ to make it $K^k_r$-free. It follows from the hypergraph removal lemma that $\ma{H}^*$ contains at least $\delta(\epsilon)n^r$ copies of $K^k_r$, where $\delta(\epsilon)$ is some constant given by the lemma.

    We proceed to show that $\ma{H}^*$ contains a copy of $K_r^k$, denoted as $\hat{K}_r^k$, which meets each $A\in\ma{H}'$ in at most $k$ vertices. It is sufficient to show that $\ma{H}^*$ contains at most $O(n^{r-1})$ (which is strictly less than $\delta(\epsilon)n^r$ for sufficiently large $n$) copies of $K_r^k$ which meet some $A\in\ma{H}'$ in at least $k+1$ vertices. Indeed, since there are at most $O(n^k)$ choices for $A\in\ma{H}'$ and at most $\binom{r}{k+1}n^{r-k-1}$ choices for $K^k_r$ which share at least $k+1$ common vertices with $A$, the above statement follows immediately.



    We conclude that there exists a $\hat{K}^{k}_r\s\ma{H}^*$ whose $\binom{r}{k}$ $k$-edges come from $\binom{r}{k}$ distinct $r$-edges of $\ma{H}'$. In particular, for $r\ge k+1\ge e$, we can take an arbitrary copy of
    $K^k_{k+1}$ that is contained in $\hat{K}^{k}_r$, and choose any $e$ of its edges, denoted by $B_1,\ldots,B_e$. Consider the $e$ edges $A_1,\ldots,A_e$ of $\ma{H}'$, which satisfy $B_i\s A_i$ for $1\le i\le e$. Note that the existence of these $e$ edges is guaranteed by the choice of $\hat{K}^{k}_r$. Observe that $B_1,\ldots,B_e$ are $e$ edges spanned by only $k+1$ vertices. Since $|A_i\setminus B_i|\le r-k$ for each $1\le i\le e$, it follows that
    $$|\cup_{i=1}^e A_i|=|\cup_{i=1}^e A_i\setminus B_i|+|\cup_{i=1}^e B_i|\le er-(e-1)k+1,$$

    \noindent which violates the assumption that $\ma{H}'$ is $\ma{G}_r(er-(e-1)k+1,e)$-free. Therefore, the theorem is proved by contradiction.
\end{proof}

\begin{proof}[\textbf{Proof of Theorem \ref{generalupperbd}}]
    Let $\ma{H}$ be an $n$-vertex $r$-graph which is $\ma{G}_r(er-(e-1)k+i+1,e)$-free. Assume towards contradiction that $|\ma{H}|\ge\epsilon n^k$ for some constant $\epsilon>0$. We follow the line of the proof of Theorem \ref{upperbd}. As in the last step of that proof, one can take an arbitrary copy of $K^k_{k+i+1}$ that is contained in $\hat{K}^{k}_r$, and choose any $e$ of its edges, denoted by $B_1,\ldots,B_e$. Consider the $e$ edges $A_1,\ldots,A_e$ of $\ma{H}'$, which satisfy $B_i\s A_i$ for $1\le i\le e$. Then, $B_1,\ldots,B_e$ are $e$ edges spanned by $k+i+1$ vertices. It follows that
    $$|\cup_{i=1}^e A_i|=|\cup_{i=1}^e A_i\setminus B_i|+|\cup_{i=1}^e B_i|\le er-(e-1)k+i+1,$$

    \noindent contradicting the assumption that $\ma{H}'$ is $\ma{G}_r(er-(e-1)k+i+1,e)$-free.
\end{proof}

\section{Rainbow cycles and solution-free sets}\label{3}





\subsection{Rainbow cycles}


  \noindent An $r$-graph $\ma{H}$ is {\it $r$-partite} if its vertex set $V(\ma{H})$ can be colored in $r$ colors in such a way that no edge of $\ma{H}$ contains two vertices of the same color. In such a coloring, the color classes of $V(\ma{H})$, i.e., the sets of all vertices of the same color, are called vertex parts of $\ma{H}$. We use $V_1,\ldots,V_r$ to denote the $r$ color classes of $V(\ma{H})$. Then $V(\ma{H})$ is a disjoint union of the $V_i$'s and $|A\cap V_i|=1$ for every $A\in\ma{H}$ and $1\le i\le r$.
  It will be convenient to use a table with $r$ rows to represent an $r$-partite $r$-graph, where the rows represent the vertex parts and the columns represent the edges. For an integer $1\le i\le r$ and a column $A\in\ma{H}$, the symbol in row $i$ and column $A$ is simply $A\cap V_i$.

  Next, we introduce the hypergraph cycles introduced by Berge \cite{hypercycle2,hypercycle1}. For $k\ge2$, a $k$-cycle in a hypergraph $\ma{H}$ is an alternating sequence of vertices and edges of the form $v_1,A_1,v_2,A_2,\ldots,v_k,A_k,v_1$ such that

  \begin{itemize}
      \item [(a)]$v_1,v_2,\ldots,v_k$ are distinct vertices of $\ma{H}$,
      \item [(b)]$A_1,A_2,\ldots,A_k$ are distinct edges of $\ma{H}$,
      \item [(c)]$v_i,v_{i+1}\in A_i$ for $1\le i\le k-1$ and $v_k,v_1\in A_k$.
  \end{itemize}

\noindent It is clear that $v_i\in A_{i-1}\cap A_{i}$ for $2\le i\le k$ and $v_1\in A_k\cap A_1$.

 Below we present the definition of rainbow cycles, which are originally introduced in \cite{shf}. 
 Let $\ma{H}$ be a linear $r$-partite $r$-graph. A $k$-cycle $v_1,A_1,v_2,A_2,\ldots,v_k,A_k,v_1$ is said to be {\it rainbow} if $v_1,\ldots,v_k$ are located in $k$ different parts of $V(\ma{H})$. It is easy to see that for $r$-partite hypergraphs, a rainbow $k$-cycle
  exists only if $k\le r$. A key ingredient of this paper is that we show in certain circumstances rainbow-cycle-free hypergraphs are also
  sparse.

 In this paper, we are mostly interested in rainbow cycles of lengths three and four. The lemma below connects hypergraphs with no rainbow 3-cycles to sparse hypergraphs.

  \begin{lemma}\label{noTim63}
    A linear $r$-partite $r$-graph $\ma{H}$ is $\ma{G}_r(3r-3,3)$-free if and only if it contains no rainbow 3-cycles.
  \end{lemma}

  \begin{proof}
    Observe that the ``only if'' part follows immediately from the definition. To prove the ``if'' part, it suffices to show that if three distinct edges $A_1,A_2,A_3\in\ma{H}$ satisfy $|A_1\cup A_2\cup A_3|\le 3r-3$, then they must also form a rainbow 3-cycle. As $\ma{H}$ is linear, by the inclusion-exclusion principle we have that
    \begin{equation*}
        \begin{aligned}
           3r-3\ge|\cup_{i=1}^3 A_i|=\sum_{i=1}^3 |A_i|-\sum_{1\le i<j\le 3}|A_i\cap A_j|+|\cap_{i=1}^3 A_i|\ge 3r-3+|\cap_{i=1}^3 A_i|,
        \end{aligned}
    \end{equation*}

    \noindent which implies that $|\cup_{i=1}^3 A_i|=3r-3$, $|\cap_{i=1}^3 A_i|=0$ and $|A_i\cap A_j|=1$ for all $1\le i<j\le 3$. Assume without loss of generality that $A_1\cap A_2=\{a\}$, $A_2\cap A_3=\{b\}$ and $A_3\cap A_1=\{c\}$. Since $a,b,c$ are distinct, it is easy to verify by the $r$-partiteness of $\ma{H}$ that they must belong to three distinct vertex parts. Assume without loss of generality that $a\in V_1$, $b\in V_2$, and $c\in V_3$. Then, we will have a rainbow 3-cycle $c,A_1,a,A_2,b,A_3,c$, as depicted by Table \ref{rain3-cycle} below, which is a  contradiction.
    \end{proof}

    \begin{table}[h]
    \begin{center}
    \begin{tabular}{|c|c|c|c|}
      \hline
       & $A_1$ & $A_2$ & $A_3$ \\\hline
      $V_1$ & $a$ & $a$ & \\\hline
      $V_2$ &  & $b$ & $b$  \\\hline
      $V_3$ & $c$ &  & $c$  \\\hline
    \end{tabular}
        \end{center}
        \caption{A rainbow 3-cycle}\label{rain3-cycle}
    \end{table}

  If a hypergraph $\{A_1,A_2,A_3,A_4\}$ forms a rainbow 4-cycle $d,A_1,a,A_2,b,A_3,c,A_4,d$, then we may represent it by Table \ref{rain4-cycle} below.

   \begin{table}[h]
  \begin{center}
    \begin{tabular}{|c|c|c|c|c|}
      \hline
       & $A_1$ & $A_2$ & $A_3$ & $A_4$ \\\hline
      $V_1$ & $a$ & $a$ & &\\\hline
      $V_2$ &  & $b$ & $b$ & \\\hline
      $V_3$ &  &  & $c$ & $c$\\\hline
      $V_4$ & $d$ &  &  & $d$\\\hline
    \end{tabular}
  \end{center}
   \caption{A rainbow 4-cycle}\label{rain4-cycle}
    \end{table}

  \subsection{Solution-free sets}

\noindent We call a linear equation $\sum_{i=1}^s a_im_i=0$ with integer coefficients $a_1,\ldots,a_s$ in the unknowns $m_i$ {\it homogeneous} if $\sum_{i=1}^s a_i=0$.
We say that $M\s[n]$ has no nontrivial solution to the equation above, if whenever
$m_i\in M$ and $\sum_{i=1}^s a_im_i=0$, it follows that all of the $m_i$'s are equal.
Note that this definition of a nontrivial solution is a simplification of the original one of Ruzsa \cite{ruzsa}.
Let $R=\{b_1,\ldots,b_r\}$ be a set of $r$ distinct nonnegative integers.
Given an integer $3\le L\le r$, a set $M\s [n]$ is said to be {\it $R_L$-solution-free} if for any integer $l\in\{3,\ldots,L\}$ and all tuples $(b_{j_1},\ldots,b_{j_l})$ of $l$ distinct elements of $R$, the equation

\begin{equation*}\label{equation-new}
    (b_{j_2}-b_{j_1})m_1+(b_{j_3}-b_{j_2})m_2+\cdots+(b_{j_l}-b_{j_{l-1}})m_{l-1}+(b_{j_1}-b_{j_l})m_l=0
\end{equation*}
\noindent has no solution in $M$ except for the trivial one $m_1=\cdots=m_l$.

\begin{remark}
The notion of $R_L$-solution-free set is a generalization of the solution-free set, which has been studied extensively in the literature (see \cite{ruzsa} for a detailed introduction). Such a generalization was first proposed in \cite{shf} for $r=4$ and $L=4$.
\end{remark}

The main objective of this subsection is to prove Theorem \ref{R4sumfree}, which says that one can actually construct an $R_4$-sum-free set $M\s[n]$ with cardinality $|M|>n^{1-o(1)}$. We will need several lemmas before presenting the proof. 

\begin{lemma}\label{nepsilon}
For any fixed constant $a\in(0,1)$, it holds that for sufficiently large $n,$ $2^{O(\log^a n)}=n^{o(1)}$.
\end{lemma}

\begin{proof}
    This lemma follows easily from the following computation. For an arbitrary small constant $\epsilon>0$, it holds that
    $$\lim_{n\rightarrow+\infty}\fr{2^{O(\log^a n)}}{n^{\epsilon}}=
    \lim_{n\rightarrow+\infty}\fr{2^{O(\log^a n)}}{2^{\epsilon\log n}}=
    \lim_{n\rightarrow+\infty}2^{O(\log^a n)-\epsilon\log n}=2^{-\Omega(\log n)}.$$
\end{proof}

The next lemma is a standard application of the Behrend-type construction \cite{Behrend} (see also \cite{ipp,ruzsa}).

  \begin{lemma}\label{4-solution-free}
    Let $l\ge 2$ be a fixed integer. Let $a_1,\ldots,a_l\in[n^{o(1)}]$ be $l$ integers (which are not necessarily fixed and can possibly be a function of $n$). Then, there exists a set $M\s [n]$ with $|M|\ge\fr{n}{2^{O\big(\sqrt{\log n\log\sum_{i=1}^l a_i}\big)}}$ which has no nontrivial solution to the equation
    \begin{equation}\label{needed}
      a_1m_1+\cdots+a_lm_l=(a_1+\cdots+a_l)m_{l+1}.
    \end{equation}
  \end{lemma}

   \begin{proof}
      Let $d$ be an integer that will be specified later and let $k=\lf\fr{\log n}{\log d}\rf$. It is easy to see that
      $$\fr{n}{d}=d^{\fr{\log n}{\log d}-1}<d^k\le d^{\fr{\log n}{\log d}}=n.$$
      Denote $D=\sum_{i=1}^l a_i$. We define the set $M\s[n]$ by
      $$M=\big\{\sum_{i=1}^k x_id^{i-1}:0\le x_i\le\lc\frac{d}{D}\rc-1\text{ and }\sum_{i=1}^k x_i^2=R\big\},$$
      where $R$ is an integer in $\{0,\ldots,k(\lc\frac{d}{D}\rc-1)^2\}$ and is chosen to maximize the size of $M$. By averaging, there exists an $R$ such that
      $$|M|>(\fr{d}{D})^k\cdot\fr{1}{k(\fr{d}{D})^2}>\fr{n}{d}\cdot\fr{1}{D^k}\cdot\fr{1}{k(\fr{d}{D})^2}=\fr{n}{kd^3D^{k-2}}.$$
      Set $d=\lf2^{\sqrt{\log n\log D}}\rf$.
      Then
      $$|M|>\fr{n}{\fr{\log n}{\log d}\cdot2^{3\sqrt{\log n\log D}}\cdot D^{\fr{\log n}{\log d}}}\ge\fr{n}{2^{O(\sqrt{\log n\log D})}},$$
      as needed.

    It suffices to show that $M$ contains no nontrivial solution to \eqref{needed}. For that purpose, let $\{m_j\in M:1\le j\le l+1\}$ be a solution to \eqref{needed}.
      By the definition of $M$ we can write $m_j=\sum_{i=1}^k x_{j,i}d^i$, where $0\le x_{j,i}\le \lc\fr{d}{D}\rc-1<\frac{d}{D}$.
      Then 
      $$\sum_{j=1}^l a_j\left(\sum_{i=1}^k x_{j,i}d^i\right)=D\sum_{i=1}^k x_{l+1,i}d^i,$$
      which implies that
      $$\sum_{i=1}^k\left(\sum_{j=1}^l a_jx_{j,i}\right)d^i=\sum_{i=1}^k Dx_{l+1,i}d^i.$$
      As $\sum_{j=1}^l a_jx_{j,i}<d$ for every $1\le i\le k$, it follows that $$\sum_{j=1}^l a_jx_{j,i}=Dx_{l+1,i}.$$
      Using the Cauchy--Schwarz inequality one can show that for $1\le i\le k$,
      $$D\left(\sum_{j=1}^l a_jx_{j,i}^2\right)=\left(\sum_{j=1}^l a_j\right)\left(\sum_{j=1}^l a_jx_{j,i}^2\right)\ge\left(\sum_{j=1}^l a_jx_{j,i}\right)^2=D^2x_{l+1,i}^2.$$
      which implies that
      \begin{equation}\label{add-1}
          \sum_{j=1}^l a_jx_{j,i}^2\ge Dx_{l+1,i}^2
      \end{equation}

      \noindent and the inequality holds when $x_{1,i}=\cdots=x_{l+1,i}$.
       On the other hand, since $\sum_{i=1}^k x_{j,i}^2=R$ for each $1\le j\le l+1$, it is straightforward to check that
       \begin{equation}\label{add-2}
          \sum_{i=1}^k\left(\sum_{j=1}^l a_jx_{j,i}^2\right)=\sum_{j=1}^la_j\left(\sum_{i=1}^kx_{j,i}^2\right)=\sum_{j=1}^la_jR=DR=D\sum_{i=1}^kx_{l+1,i}^2.
      \end{equation}

      \noindent Combining \eqref{add-1} and \eqref{add-2}, we conclude that for each $1\le i\le k$, we have that
      $x_{1,i}=\cdots=x_{l+1,i}$, which further implies that $m_1=\cdots=m_{l+1}$. Therefore, $M$ contains only trivial solutions to \eqref{needed}, as needed.
   \end{proof}

We will need three more technical lemmas. The lemma below can be viewed as a more involved application of the Behrend-type construction and it generalizes also the solution-free set constructed in \cite{ipp}.

   \begin{lemma}\label{sumfree}
    Let $a_1,a_2,a_3,a_4\in[n^{o(1)}]$ be four integers (which are not necessarily fixed and can possibly be a function of $n$) satisfying
    \begin{itemize}
      \item [\rm{(i)}] $a_1<a_2<a_3<a_4$,
      \item [\rm{(ii)}] $a_1+a_4=a_2+a_3,$
      \item [\rm{(iii)}] $a_1,a_2=a_3^{o(1)}$, $a_4=a_3+a_3^{o(1)}$,
      \item [\rm{(iv)}] there are two constants $0<a,b<1$ satisfying $\log a_2=O(\log^a a_3)$ and $\log a_3=\Omega(\log^b n)$.
    \end{itemize}
    Then there exists a set $M\s [n]$ with $|M|\ge\fr{n}{2^{O\big((\log n)^{1+\fr{b(a-1)}{2}}\big)}}>n^{1-o(1)}$ which contains no nontrivial solution to the equation
        \begin{equation}\label{needed-2}
  a_1x+a_4y=a_2u+a_3v.
    \end{equation}
   \end{lemma}

   \begin{remark}
        Observe that $0<1+\fr{b(a-1)}{2}<1$ as $0<a,b<1$. Therefore, by Lemma \ref{nepsilon} we have $2^{O\big((\log n)^{1+\fr{b(a-1)}{2}}\big)}=n^{o(1)}$ for sufficiently large $n$.
   \end{remark}

   \begin{proof}
        Let $\ma{B}\s\{1,\ldots,\lf\fr{a_3+1}{a_2}\rf\}$ be a set of integers which has no nontrivial solution to the auxiliary equation
        \begin{equation}\label{auxi}
        \begin{aligned}
            a_1x+(a_4-a_3-1)y+v=a_2u.
        \end{aligned}
    \end{equation}
    \noindent Observe that by (ii) and (iii) we have $a_1+(a_4-a_3-1)+1=a_2$ and $a_1,a_4-a_3,a_2\in[(\frac{a_3}{a_2})^{o(1)}]$. Then it follows by Lemma \ref{4-solution-free} that for sufficiently large $\frac{a_3}{a_2}$, there exists a subset $\ma{B}$ which satisfies

     \begin{equation}\label{add-3}
   |\ma{B}|\ge\fr{\fr{a_3+1}{a_2}}{2^{O\big(\sqrt{\log \fr{a_3+1}{a_2}\log a_2}\big)}}=\fr{a_3+1}{2^{\log a_2+O\big(\sqrt{\log \fr{a_3+1}{a_2}\log a_2}\big)}}>\fr{a_3}{2^{O\big((\log a_3)^{\fr{1+a}{2}}\big)}}>a_3^{1-o(1)},
    \end{equation}


    \noindent where the first, second and third inequalities follow from Lemma \ref{4-solution-free}, (iv) and Lemma \ref{nepsilon}, respectively. Let $M$ be the family of all integers in $[n]$ whose representations in base $a_3+1$ contain only digits belonging to $\ma{B}$, namely,
    $$M=\big\{m\in[n]:m=\sum_i u_i(a_3+1)^i,~u_i\in\ma{B}\big\}.$$
    It is not hard to see that

    \begin{equation}\label{add-4}
      |M|\ge|\ma{B}|^{\lf\log_{a_3+1}n\rf}=\Omega\big(|\ma{B}|^{\fr{\log n}{\log(a_3+1)}-1}\big)=\Omega\big(n^{\fr{\log|\ma{B}|}{\log(a_3+1)}-\fr{\log |\ma{B}|}{\log n}}\big).
    \end{equation}

    \noindent As $\fr{\log |\ma{B}|}{\log n}=o(1)$, substituting the second inequality of \eqref{add-3} into \eqref{add-4} one gets

    $$|M|\ge n^{\fr{\log a_3-O\big((\log a_3)^{\fr{1+a}{2}}\big)}{\log(a_3+1)}}\ge n^{1-O\big((\log a_3)^{\fr{a-1}{2}}\big)}=\fr{n}{2^{O\big((\log n)^{1+\fr{b(a-1)}{2}}\big)}},$$
    \noindent as needed, where the equality follows from (iv) and the fact that $n^{O(\log^{(a-1)/2} a_3)}=2^{O(\log n\log^{(a-1)/2} a_3)}$.

Therefore, to prove the lemma it remains to show that \eqref{needed-2} has no nontrivial solution in $M$.
    Assume towards contradiction that $x,y,u,v\in M$ form a nontrivial solution to the equation above. Let us represent them in base $a_3+1$ as
    $$x=\sum_i x_i(a_3+1)^i,~y=\sum_i y_i(a_3+1)^i,~u=\sum_i u_i(a_3+1)^i\text{ and }v=\sum_i v_i(a_3+1)^i.$$

    \noindent Since $x,y,u,v$ form a nontrivial solution, there must exist some integer $i$ such that $x_i,y_i,u_i,v_i$ are not all equal. Let $j$ be the least $i$ satisfying such a condition. Then, it is not hard to check that
    $$a_1x_j(a_3+1)^j+a_4y_j(a_3+1)^j\equiv a_2u_j(a_3+1)^j+a_3v_j(a_3+1)^j\pmod{(a_3+1)^{j+1}},$$

    \noindent which implies that $$a_1x_j+(a_4-a_3-1)y_j\equiv a_2u_j-v_j\pmod{a_3+1},$$

    \noindent or equivalently,
    $$a_1x_j+(a_4-a_3-1)y_j+v_j\equiv a_2u_j\pmod{a_3+1}.$$

    \noindent Recall that by definition we have that $x_j,y_j,u_j,v_j<\fr{a_3+1}{a_2}$. Then by the equation above,
    $$a_1x_j+(a_4-a_3-1)y_j+v_j=a_2u_j,$$

    \noindent which is a contradiction since $x_j,y_j,u_j,v_j$ are not all equal and $\ma{B}$ has no nontrivial solution to (\ref{auxi}).
   \end{proof}


   \begin{lemma}\label{shift}
    Let $s\ge 3$ be a fixed integer and $\sum_{i=1}^s a_im_i=0$ be a homogeneous linear equation with the unknowns $m_i$. If $M\s [n]$ has no nontrivial solution to this equation, then so does for any shift $(M+b)\cap[n]$, where $b\in\mathbb{Z}$ and $M+b:=\{m+b:m\in M\}$.
   \end{lemma}

   \begin{proof}
     Assume towards contradiction that there exists some $b\in\mathbb{Z}$ such that $(M+b)\cap[n]$ contains a nontrivial solution to the equation. Denote this nontrivial solution by $\{b_1,\ldots,b_s\}$, where for each $1\le i\le s$, $b_i=m_i+b$ and $m_i\in M$. Since $b_1,\ldots,b_s$ are not all equal, neither are $m_1,\ldots,m_s$. Therefore,
     $$0=\sum_{i=1}^s a_ib_i=\sum_{i=1}^s a_i(m_i+b)=\sum_{i=1}^s a_im_i+b\sum_{i=1}^s a_i=\sum_{i=1}^s a_im_i.$$
     \noindent We conclude that $\{m_1,\ldots,m_s\}\s M$ also forms a nontrivial solution, contradicting the assumption of the lemma.
   \end{proof}

   \begin{lemma}\label{prmethod}
       Let $0<a<1$ be a fixed constant and $t$ be a fixed positive integer. Let $\sum_{i=1}^s a_{ij}m_i=0$, $1\le j\le t$, be $t$ homogeneous linear equations with the unknowns $m_i$. If for every integer $1\le j\le t$, there exists a set $M_j\s[n]$ with  $|M_j|\ge\fr{n}{2^{O(\log^a n)}}$ which contains no nontrivial solution to the equation $\sum_{i=1}^s a_{ij}m_i=0$. Then, there exists a set $M\s[n]$ with $|M|\ge\fr{n}{2^{O(\log^a n)}}$ which contains no nontrivial solution to any of the $t$ equations $\sum_{i=1}^s a_{ij}m_i=0$, $1\le j\le t$.
   \end{lemma}

   \begin{proof}
    Choose $t-1$ integers $\mu_2,\ldots,\mu_t\in\{-n,\ldots,n\}$ uniformly and independently at random. It follows by Lemma \ref{shift} that $M=M_1\cap\big(\cap_{j=2}^t (M_j+\mu_j)\big)$ has no nontrivial solution to any of the $t$ equations. So it remains to prove the required lower bound on $|M|$. For that purpose, let us compute the probability that an arbitrary element $m\in M_1$ lies in the intersection $\cap_{j=2}^t (M_j+\mu_j)$. Observe that $-n\le m-m_j\le n$ holds for any $2\le j\le t$ and any $m_j\in M_j$. It follows that for any $m\in M_1$, 
    $$\Pr[m\in(M_j+\mu_j)]=\Pr[\exists~m_j\in M_j,~\text{s.t.}~\mu_j=m-m_j]=\fr{|M_j|}{2n+1}=2^{-O(\log^a n)}.$$

    \noindent Since $\{m\in M_j+\mu_j:2\le j\le t\}$ is a set of mutually independent events, 
    $$\Pr[m\in\cap_{j=2}^t (M_j+\mu_j)]=(2^{-O(\log^a n)})^(t-1)=2^{-O(\log^a n)}.$$

    \noindent By the linearity of expectation, 
    $${\rm E}[|M|]\ge|M_1|\cdot 2^{-O(\log^a n)}=\fr{n}{2^{O(\log^a n)}}.$$

    \noindent Consequently, there exist a choice of $t-1$ integers $\mu_2,\ldots,\mu_t\in\{-n,\ldots,n\}$ such that $M=M_1\cap\big(\cap_{j=2}^t (M_j+\mu_j)\big)$ is a set of at least $\fr{n}{2^{O(\log^a n)}}$ positive integers which has no nontrivial solution to any of the $t$ equations, as needed.
   \end{proof}

   The following theorem is the main result of this subsection.

   \begin{theorem}\label{R4sumfree}
    For every fixed integer $r\ge 4$, there exists an $r$-element set $R\s[n]$ and an $R_4$-solution-free set $M\s[n]$ with $|M|>n^{1-o(1)}$.
   \end{theorem}

   \begin{proof}
    In order to construct the required solution-free set, we typically choose a set $R=\{b_1,\ldots,b_r\}$ that satisfies $\log b_1=\lf(\log n)^{\frac{1}{2^r}}\rf$ and for $2\le i\le r$, $\log b_i=\lf\log^2 b_{i-1}\rf$. Then $\log^2 b_{i-1}-1<\log b_i\le\log^2 b_{i-1}$, which implies that $b_{i-1}\le 2^{\sqrt{\log b_i+1}}$ and $b_{i-1}=b_i^{o(1)}$ (by Lemma \ref{nepsilon}). It is also not hard to check that
    $$\log^{\frac{1}{2^{r-i+1}}}n-o(\log^{\frac{1}{2^{r-i+1}}}n)\le\log b_i\le\log^{2^{i-1}} b_1\le\log^{\frac{1}{2^{r-i+1}}} n,$$

    \noindent which implies that $\max\{b_i:1\le i\le r\}=b_r\le 2^{\sqrt{\log n}}$ and $b_r=n^{o(1)}$ (also by Lemma \ref{nepsilon}).



    It is routine to check by definition that a subset $M\s[n]$ is $R_4$-solution-free if and only if for any tuple $(b_{j_1},b_{j_2},b_{j_3},b_{j_4})$ of 4 distinct elements of $R$, where we assume without loss of generality that $b_{j_1}<b_{j_2}<b_{j_3}<b_{j_4}$, $M$ contains no nontrivial solution to any of the following four equations:


    \begin{equation*}\label{type1}
    \text{Type 1.}\qquad\qquad\qquad\qquad\quad      (b_{j_2}-b_{j_1})m_1+(b_{j_3}-b_{j_2})m_2=(b_{j_3}-b_{j_1})m_3,
    \end{equation*}

    \begin{equation*}\label{type2}
    \text{Type 2.}\qquad    (b_{j_2}-b_{j_1})m_1+(b_{j_3}-b_{j_2})m_2+(b_{j_4}-b_{j_3})m_3=(b_{j_4}-b_{j_1})m_4,
    \end{equation*}

    \begin{equation*}\label{type3}
    \text{Type 3.}\qquad    (b_{j_2}-b_{j_1})m_1+(b_{j_4}-b_{j_2})m_2=(b_{j_4}-b_{j_3})m_3+(b_{j_3}-b_{j_1})m_4,
    \end{equation*}

    \begin{equation*}\label{type4}
    \text{Type 4.}\qquad    (b_{j_4}-b_{j_1})m_1+(b_{j_3}-b_{j_2})m_2=(b_{j_3}-b_{j_1})m_3+(b_{j_4}-b_{j_2})m_4.
    \end{equation*}

    \noindent 
    Since there are $\binom{r}{3}$ Type 1 equations and $\binom{r}{4}$ Type $i$ equations for each $i=2,3,4$, we conclude that $M\s[n]$ is $R_4$-solution-free with respect to $R=\{b_1,\ldots,b_r\}$ if and only if it contains no nontrivial solution to any of the $t:=\binom{r}{3}+3\binom{r}{4}$ equations described above. For notational convenience, let us denote those equations by $Eq_1,\ldots,Eq_t$, respectively.


    In order to construct a sufficiently large solution-free set $M\s[n]$, in the rest of the proof we proceed as follows. Using Lemmas \ref{4-solution-free} and \ref{sumfree} we can construct $t$ sets $M_1,\ldots,M_t\s[n]$ such that for $1\le i\le t$, $|M_i|>n^{1-o(1)}$ and it has no nontrivial solution to the equation $Eq_i$. Then Lemma \ref{prmethod} will guarantee the existence of a large set $M:=M_1\cap\big(\cap_{j=2}^t (M_j+\mu_j)\big)$ which contains no nontrivial solution to any of the equations $Eq_1,\ldots,Eq_t$, as detailed below. 

    Since $\max\{b_i:1\le i\le r\}=b_r=n^{o(1)}$, applying Lemma \ref{4-solution-free} with $l=3,4$, it follows that for every equation of Types 1 or 2, there exists a set $M'\s[n]$,
    $$|M'|\ge\fr{n}{2^{O(\sqrt{\log n\log b_{r}})}}=\fr{n}{2^{O(\log^{3/4} n)}}$$

    \noindent with no nontrivial solution to it.

    It remains to consider equations of Types 3 and 4. Let us compare the coefficients of a Type 3 equation with that of Lemma \ref{sumfree}. Set $a_1:=b_{j_2}-b_{j_1}$, $a_2:=b_{j_3}-b_{j_1}$, $a_3:=b_{j_4}-b_{j_3}$ and $a_4:=b_{j_4}-b_{j_2}$. It is not hard to check that $a_i,~1\le i\le 4$ satisfy the four constraints of Lemma \ref{sumfree}, with  
    $a\le\fr{1}{2}$ and $b\ge\fr{1}{2^{r-3}}$, where the inequalities hold when $j_4=j_3+1$ and $j_4=4$, respectively.
    Thus for every Type 3 equation, there exists a set $M''\s[n]$,
     $$|M''|\ge\fr{n}{2^{O\big((\log n)^{1+\fr{b(a-1)}{2}}\big)}}\ge\fr{n}{2^{O\big((\log n)^{1-\fr{1}{2^{r-1}}}\big)}}$$
     \noindent with no nontrivial solution to it, where the second inequality is obtained by setting $a=\frac{1}{2}$ and $b=\frac{1}{2^{r-3}}$.

     For a Type 4 equation, we can prove similar results by setting $a_1:=b_{j_3}-b_{j_2}$, $a_2:=b_{j_3}-b_{j_1}$, $a_3:=b_{j_4}-b_{j_2}$ and $a_4:=b_{j_4}-b_{j_1}$. The details are omitted.

    Observe that $c:=1-\fr{1}{2^{r-1}}>\frac{3}{4}$ holds for any $r\ge 4$. Then according to the discussions above, for $1\le i\le t$, there exists a set $M_i\s[n]$ with $|M_i|\ge\fr{n}{2^{O(\log^c n)}}$ which contains no nontrivial solution to $Eq_i$. By Lemma \ref{prmethod}, there exists a set $M\s[n]$ with $|M|\ge\fr{n}{2^{O(\log^c n)}}$ which contains no nontrivial solution to any of the $t$ equations $Eq_1,\ldots,Eq_t$, that is, $M\s[n]$ is an $R_4$-solution-free set with $R=\{b_1,\ldots,b_r\}$. Moreover, we have $|M|>n^{1-o(1)}$ by Lemma \ref{nepsilon}. The proof of the theorem is thus completed.
   \end{proof}

   The proof of Theorem \ref{R4sumfree} in fact leads to the following result, whose proof is omitted.

   \begin{theorem}\label{R3sumfree}
    For every fixed integer $r\ge 3$, there exists an $r$-element set $R\s[n]$ and an $R_3$-solution-free set $M\s[n]$ with $|M|>n^{1-o(1)}$.
   \end{theorem}

\section{Using solution-free sets to construct rainbow-cycle-free hypergraphs}\label{4}

 \noindent It is known that tools from additive number theory can be used to construct hypergraphs satisfying certain Tur{\'a}n-type properties, see, e.g. \cite{ipp,Furediconst,RS,shf}. Given a positive integer $r$ and an appropriate solution-free set $M\s [n]$, we can construct an $r$-partite $r$-graph $\ma{H}_M$ as follows. The vertex set $V(\ma{H}_M)$ is the disjoint union of $r$ pairwise disjoint vertex parts $V_1,\ldots,V_r$, i.e., $V(\ma{H}_M)=\cup_{i=1}^r V_i$, where for each $1\le i\le r$, $V_i$ is a copy of $[n^{1+o(1)}]$.
 The edge set is defined as

   \begin{equation*}
       \ma{H}_M=\left\{A(y,m)=(y+b_1m,\ldots,y+b_rm):y\in[n],m\in M\right\}\s V_1\times\cdots\times V_r,
   \end{equation*}

   \noindent where $R:=\{b_1,b_2,\ldots,b_r\}\s[n^{o(1)}]$ is a set of $r$ distinct elements, and $A(y,m)$ is an ordered $r$-tuple such that $y+b_im\in V_i$ for each $1\le i\le r$.


   \begin{remark}\label{cardinality}
   Noting that the $\ma{H}_M$ is well defined as for all $y\in[n],m\in[M],1\le i\le r$, we have $y+b_im\le n^{1+o(1)}$.
    It is easy to see that $|V(\ma{H}_M)|=n^{1+o(1)}$ and $|\ma{H}_M|=n|M|$. Moreover, if $|M|>n^{1-o(1)}$, then $|\ma{H}_M|>n^{2-o(1)}>|V(\ma{H}_M)|^{2-o(1)}$ for sufficiently large $n$ and suitable $o(1)$ (where $o(1)\rightarrow0$ as $n\rightarrow\infty$).
    \end{remark}

    The following lemma is easy to prove.

   \begin{lemma}\label{linear}
       For any $M\s[n]$ and $R\s[n^{o(1)}]$, $\ma{H}_M$ is always a linear hypergraph. 
   \end{lemma}

   \begin{proof}
    Assume towards contradiction that there exist $y,y'\in [n]$ and $m,m'\in M$ such that $(y,m)\neq(y',m')$ and $|A(y,m)\cap A(y',m')|\ge 2$. Then, there exist $1\le i,j\le r$ and $i\neq j$ such that
         \begin{equation*}
    \left\{
        \begin{aligned}
           y+b_im&=y'+b_im',\\
           y+b_jm&=y'+b_jm'.\\
        \end{aligned}
    \right.
    \end{equation*}

    \noindent Since $b_i-b_j\neq 0$, it is easy to check that $y-y'=b_i(m'-m)=b_j(m'-m)$, which implies that $m=m'$, and hence $y=y'$, a contradiction.
   \end{proof}

   The next theorem connects solution-free sets to rainbow-cycle-free hypergraphs.

   \begin{theorem}\label{sumfree-rainbow}
        Given $M\s[n]$ and $R\s[n^{o(1)}]$, if $M$ is an $R_L$-solution-free set for some $3\le L\le r$, then the hypergraph $\ma{H}_M$ constructed above is a linear $r$-partite $r$-graph which contains no rainbow cycle of length less than $L+1$.
   \end{theorem}

   \begin{proof}
       By our construction $\ma{H}_M$ is clearly an $r$-partite $r$-graph. Moreover, its linearity is justified by Lemma \ref{linear}. It remains to show that $\ma{H}_M$ does not contain any rainbow cycle of length less than $L+1$.

       Assume towards contradiction that $\ma{H}_M$ contains a rainbow $l$-cycle for some $3\le l\le L$. Denote such a cycle by
       $$v_1,A(y_1,m_1),v_2,A(y_2,m_2),\ldots,v_l,A(y_l,m_l),v_1,$$

      \noindent where we assume that there exist $l$ distinct integers $1\le j_1,\ldots,j_l\le r$ such that $v_i\in V_{j_i}$ for $1\le i\le l$. 
       According to the linearity and the $r$-partiteness of $\ma{H}_M$, it is easy to check by definition that
      $$A(y_i,m_i)\cap A(y_{i+1},m_{i+1})\cap V_{j_{i+1}}=\{v_{i+1}\}\text{ for $1\le i\le l-1$ and }A(y_l,m_l)\cap A(y_1,m_1)\cap V_{j_1}=\{v_1\}.$$

      \noindent
      As for $1\le j\le r$, $A(y,m)\cap V_j=\{y+b_jm\}$, the following $l$ equations must hold simultaneously
      $$ y_i+b_{j_{i+1}}m_i=y_{i+1}+b_{j_{i+1}}m_{i+1}\text{\quad for $1\le i\le l-1$ \quad and\quad  $y_l+b_{j_1}m_l=y_1+b_{j_1}m_1$}.$$

    \noindent By a simple elimination, one can infer that
    $$(b_{j_2}-b_{j_1})m_1+(b_{j_3}-b_{j_2})m_2+\cdots+(b_{j_l}-b_{j_{l-1}})m_{l-1}+(b_{j_1}-b_{j_l})m_l=0.$$
    \noindent We conclude that $m_1=\cdots=m_l$, since $M$ is $R_L$-solution-free and $3\le l\le L$. It follows that $y_1=\cdots=y_l$ and hence $A(y_1,m_1)=\cdots=A(y_l,m_l)$, which is obviously a contradiction. Consequently, $\ma{H}_M$ contains no rainbow cycles with length less than $L+1$, as needed.
   \end{proof}


  \begin{proof}[\textbf{Proof of Theorem \ref{rainbow}}]
    This theorem is a direct consequence of Theorems \ref{R4sumfree}, \ref{R3sumfree}, \ref{sumfree-rainbow} and Remark \ref{cardinality}.
  \end{proof}

  \section{Using rainbow-cycle-free hypergraphs to construct sparse hypergraphs}\label{5}

  \noindent In this section, we will construct $\ma{G}_r(er-2e+3,e)$-free $r$-graphs for $e=4,5,7,8$. 
  The constructions for $\ma{G}_r(4r-5,4)$-free and $\ma{G}_r(5r-7,5)$-free hypergraphs are relatively simple. Indeed, below we will show that linear $\ma{G}_r(3r-3,3)$-free hypergraphs are simultaneously $\ma{G}_r(4r-5,4)$-free and $\ma{G}_r(5r-7,5)$-free. However, for $e=6$,
$\ma{G}_r(3r-3,3)$-free hypergraphs are not necessarily $\ma{G}_r(6r-9,6)$-free. In fact, we are not able to construct $\ma{G}_r(6r-9,6)$-free hypergraphs which meet the lower bound of Conjecture \ref{conjecture}. Surprisingly, for $e=7,8$, based on the rainbow-cycle-free hypergraphs, we can construct sufficiently large $\ma{G}_r(7r-11,7)$-free and $\ma{G}_r(8r-13,8)$-free hypergraphs whose cardinality match the lower bounds of Conjecture \ref{conjecture}. 

Before presenting those constructions  let us begin with several lemmas which are very useful to our proof.

  \begin{lemma}\label{applicationof63}
    Assume that $\ma{H}$ is a linear $\ma{G}_r(3r-3,3)$-free $r$-graph. Let $A$ and $B$ be two edges of $\ma{H}$ such that $A\cap B\neq\emptyset$. If some other edge $C\in\ma{H}\setminus\{A,B\}$ has nonempty intersection with both $A$ and $B$, then we must have $C\cap A=C\cap B=A\cap B$, i.e., $A,B$ and $C$ share precisely one common vertex.
  \end{lemma}

  \begin{proof}
    The lemma follows directly from the fact that $\ma{H}$ is linear and $\ma{G}_r(3r-3,3)$-free.
  \end{proof}

  \begin{lemma}\label{e-1toe}
  Let $e\ge 4$ be a positive integer and $\ma{F}=\{A_1,\ldots,A_e\}$ be a linear $r$-partite $r$-graph formed by exactly $e$ edges. Assume that $\ma{F}$ is $\ma{G}_r(3r-3,3)$-free and $\ma{G}_r((e-1)r-2(e-1)+3,e-1)$-free but not  $\ma{G}_r(er-2e+3,e)$-free, i.e., $|V(\ma{F})|\le er-2e+3$.
  Then, 
          for any $A_i\in\ma{F}$ there exist three distinct edges $A_{i_1},A_{i_2},A_{i_3}\in\ma{F}\setminus\{A_i\}$ such that
          \begin{itemize}
            \item  [\rm{(1)}]  $A_{i}$ intersects each of $A_{i_1},A_{i_2}$ and $A_{i_3}$ in a different vertex, i.e., $|A_i\cap(A_{i_1}\cup A_{i_2}\cup A_{i_3})|=3$,
            \item  [\rm{(2)}] $A_{i_1},A_{i_2}$ and $A_{i_3}$ are pairwise disjoint,
            \item  [\rm{(3)}] $|A_{i_1}\cup A_{i_2}\cup A_{i_3}\cup A_{i}|=4r-3$.
          \end{itemize}
  \end{lemma}

  \begin{proof}
    By the $\ma{G}_r((e-1)r-2(e-1)+3,e-1)$-free property of $\ma{F}$, we have $|\cup_{i=1}^{e-1}A_i|\ge (e-1)r-2(e-1)+4$. Denote $X=\cup_{i=1}^{e-1}A_i$. Below let us prove the lemma for $i=e$. Observe that
    $$er-2e+3\ge|X\cup A_{e}|=|X|+|A_{e}|-|X\cap A_{e}|\ge (e-1)r-2(e-1)+4+r-|X\cap A_{e}|,$$

    \noindent which implies that $|A_{e}\cap X|\ge 3$. Since $\ma{F}$ is linear, there exist three distinct edges $A_{i_1}, A_{i_2}, A_{i_3}\in\{A_1,\ldots,A_{e-1}\}$ such that $A_{e}$ intersects each of them in a different vertex, completing the proof of (1).

    Assume without loss of generality that $i_1=1,~i_2=2,~i_3=3$ and $A_1\cap A_{e}=\{a\}$, $A_2\cap A_{e}=\{b\}$, $A_3\cap A_{e}=\{c\}$, where $a,b,c\in V(\ma{F})$. Since $\ma{F}$ is also $r$-partite, $a,b,c$ must be located in different vertex parts of $\ma{F}$, say, $a\in V_1$, $b\in V_2$ and $c\in V_3$.
    We claim that $A_1,A_2$ and $A_3$ are pairwise disjoint. Assume towards contradiction that $A_1\cap A_2=\{d\}\neq\emptyset$. Then, by the linearity of $\ma{F}$ it is easy to see that $d\not\in\{a,b,c\}$. Since $A_1\cap A_{e}=\{a\}, A_2\cap A_{e}=\{b\}$ and $A_1\cap A_2=\{d\}$, we conclude that $|A_1\cup A_2\cup A_{e}|\le 3r-3$, contradicting the assumption that $\ma{F}$ is $\ma{G}_r(3r-3,3)$-free. Similarly, one can show that $A_1\cap A_3=\emptyset$ and $A_2\cap A_3=\emptyset$. The proof of (2) is thus completed.

    It remains to prove (3). Indeed, this statement follows from the fact that $|A_1\cup A_2\cup A_3|=3r$ and $|A_{e}\cap(A_1\cup A_2\cup A_3)|=3$. 
  \end{proof}

  \begin{lemma}\label{maxdeg}
     With the assumption of Lemma \ref{e-1toe}, it holds that for any vertex $a\in V(\ma{F})$ we have $\deg(a)\le\lf\fr{e}{3}\rf$, where $\deg(a)$ is the number of edges of $\ma{F}$ containing $a$.
  \end{lemma}

  \begin{proof}
    Suppose that $\max\{\deg(v):v\in V(\ma{F})\}=l$. To prove the lemma, it is sufficient to show that $l\le\lf\fr{e}{3}\rf$. Choose a vertex $a\in V(\ma{F})$ so that $\deg(a)=l$. Assume that $a\in V_1$ and let $A_1,\ldots,A_l$ be the $l$ edges of $\ma{F}$ that contain $a$. According to the linearity of $\ma{F}$, it is easy to see that $A_1\setminus\{a\},\ldots,A_l\setminus\{a\}$ are pairwise disjoint.
    For $1\le i\le l$, applying Lemma \ref{e-1toe} to each $A_i$, it follows that for every $A_i$ there exist three edges $B_{i_1},B_{i_2},B_{i_3}$ which satisfy the conclusion of Lemma \ref{e-1toe}.

    Again, by the linearity of $\ma{F}$, one can infer that for each $i$, at most one of the edges in $\{A_1,\ldots,A_l\}\setminus\{A_i\}$ can play the role of some edge in $\{B_{i_1},B_{i_2},B_{i_3}\}$. As a consequence, for each $1\le i\le l$, there exist at least two distinct edges, say, $B_{i_1},B_{i_2}$, belonging to $\ma{F}\setminus\{A_1,\ldots,A_l\}$ such that
    $$B_{i_j}\cap A_i\neq\emptyset,~B_{i_j}\cap A_i\neq\{a\}\text{ for $j=1,2$. }$$

    We claim that the $2l$ edges $\{B_{i_j}:1\le i\le l,~1\le j\le 2\}$ are pairwise distinct. Assume the opposite. Then there exist $1\le i\neq i'\le l$ and some $B\in\{B_{i_j}:1\le i\le l,~1\le j\le 2\}$ such that $B\cap A_i\neq\emptyset$ and $B\cap A_{i'}\neq\emptyset$. Thus by Lemma \ref{applicationof63} the only possible situation is that $B\cap A_i\cap A_{i'}=\{a\}$, a contradiction.

    Now the $A_i$'s and the $B_{i_j}$'s have brought us at least $3l$ distinct edges of $\ma{F}$, which implies that $3l\le e$ and hence $l\le\lf\fr{e}{3}\rf$, as needed. 
  \end{proof}

  \subsection{$\ma{G}_r(4r-5,4)$-free and $\ma{G}_r(5r-7,5)$-free hypergraphs}

  The main task of this subsection is to show that if a linear $r$-partite $r$-graph is $\ma{G}_r(3r-3,3)$-free, then it is also $\ma{G}_r(4r-5,4)$-free and $\ma{G}_r(5r-7,5)$-free. 

  \begin{theorem}\label{3r-3im4r-5}
    Let $\ma{H}$ be a linear $r$-partite $r$-graph. If $\ma{H}$ is $\ma{G}_r(3r-3,3)$-free, then it is also $\ma{G}_r(4r-5,4)$-free.
  \end{theorem}

  \begin{proof}
    Assume towards contradiction that $\ma{H}$ is not $\ma{G}_r(4r-5,4)$-free. Then, there exist four distinct edges $A_1,A_2,A_3,A_4\in\ma{H}$ such that $|\cup_{i=1}^4 A_i|\le 4r-5$. Applying Lemma \ref{e-1toe} with $\ma{F}=\{A_1,A_2,A_3,A_4\}$, it follows that $|\cup_{i=1}^4 A_i|=4r-3$, a contradiction.
  \end{proof}

  \begin{theorem}\label{3r-3im5r-7}
    Let $\ma{H}$ be a linear $r$-partite $r$-graph. If $\ma{H}$ is $\ma{G}_r(3r-3,3)$-free, then it is also $\ma{G}_r(5r-7,5)$-free.
  \end{theorem}

  \begin{proof}
        Assume towards contradiction that $\ma{H}$ is not $\ma{G}_r(5r-7,5)$-free. Then, there exist five distinct edges $A_1,A_2,A_3,A_4,A_5\in\ma{H}$ such that $|\cup_{i=1}^5 A_i|\le 5r-7$. By Theorem \ref{3r-3im4r-5}, $\ma{H}$ is $\ma{G}_r(4r-5,4)$-free. Applying Lemma \ref{e-1toe} to those five edges, it follows that they
        must contain at least three vertices of degree two (for example, vertices $a,b,c$ given by the proof Lemma \ref{e-1toe}). However, this  is impossible since by Lemma \ref{maxdeg} the maximum degree of the vertices contained in $\cup_{i=1}^5 A_i$ cannot exceed $\lf\fr{5}{3}\rf=1$. We conclude that $\ma{H}$ must be $\ma{G}_r(5r-7,5)$-free, as needed.
        \end{proof}

  \subsection{Classification of hypergraphs which are not $\ma{G}_r(6r-9,6)$-free}

 \noindent Consider the following linear 3-partite 3-graph (see Table \ref{tab63not96} below) formed by six edges $A_1,\ldots,A_6$ and nine vertices $\{a,b,c,d,e,f,g,h,i\}$.
 It is straightforward to check that such a 3-graph is $\ma{G}_3(6,3)$-free, $\ma{G}_3(7,4)$-free and $\ma{G}_3(8,5)$-free, but it is not $\ma{G}_3(9,6)$-free.
 It is clear that this hypergraph provides an example which illustrates that a linear $\ma{G}_r(3r-3,3)$-free $r$-partite $r$-graph is not necessarily $\ma{G}_r(6r-9,6)$-free.

    \begin{table} [ht]
    \begin{center}
    \begin{tabular}{|c|c|c|c|c|c|c|}
      \hline
       & $A_1$ & $A_2$ & $A_3$ & $A_4$ & $A_5$ & $A_6$\\\hline
      $V_1$  & $a$  & $f$  & $h$&  $f$ & $h$ & $a$ \\\hline
      $V_2$ & $d$  & $b$ & $i$ &  $i$ & $d$ & $b$\\\hline
      $V_3$ & $e$  & $g$ & $c$ &  $e$ & $g$ & $c$\\\hline
    \end{tabular}
        \end{center}
        \caption{A linear 3-partite 3-graph of six edges which is (6,3)-free but not (9,6)-free}\label{tab63not96}
    \end{table}

To our surprise, if an $r$-graph is simultaneously $\ma{G}_r(3r-3,3)$-free and rainbow-cycle-free, then even it is not $\ma{G}_r(6r-9,6)$-free,
the next theorem shows that for any six edges whose union contains at most $6r-9$ vertices, there is only one possible configuration (up to isomorphism).

 \begin{theorem}\label{clafy6r-9}
    Let $r\ge 3$ and $\ma{H}$ be a linear $r$-partite $r$-graph. Assume that $\ma{H}$ contains no rainbow cycles of length three or four. If there exist six edges $A_1,\ldots,A_6$ of $\ma{H}$ such that $|\cup_{i=1}^6 A_i|\le 6r-9$, then $|\cup_{i=1}^6 A_i|= 6r-9$ and $A_1,\ldots,A_6$
    have only one possible configuration (up to isomorphism), as described below:

     \begin{itemize}
       \item [\rm{(1)}] $A_1,A_2,A_3$ are pairwise disjoint and so are $A_4,A_5,A_6$;
       \item [\rm{(2)}] $|(A_1\cup A_2\cup A_3)\cap (A_4\cup A_5\cup A_6)|=9$ and for any $i,i'\in\{1,2,3\}$ and $j,j'\in\{4,5,6\}$, we have $|A_i\cap A_j|=1$ and $A_i\cap A_j\neq A_{i'}\cap A_{j'}$ for $(i,j)\neq (i',j')$;
       \item [\rm{(3)}]  all nine vertices in the intersection $(A_1\cup A_2\cup A_3)\cap (A_4\cup A_5\cup A_6)$ must be contained in three vertex parts, say, $V_1,V_2$ and $V_3$.
     \end{itemize}

     \noindent More precisely, the intersection $(A_1\cup A_2\cup A_3)\cap (A_4\cup A_5\cup A_6):=\{a,b,c,d,e,f,g,h,i\}$ can be characterized by Table \ref{tab63not96}. 
 \end{theorem}


 \begin{remark}
    Since $\ma{H}$ contains no rainbow 3-cycles, by Lemma \ref{noTim63} and Theorem \ref{3r-3im5r-7} it is also $\ma{G}_r(3r-3,3)$-free and $\ma{G}_r(5r-7,5)$-free. Thus it follows from Lemma \ref{maxdeg} that the maximum degree (with respect to the subhypergraph formed by $\{A_1,\ldots,A_6\}$) of a vertex $x\in\cup_{i=1}^6 A_i$ is at most two. Assume that $\{A_1,\ldots,A_6\}$ contains $\lambda$ vertices of degree two and $\mu$ vertices of degree one. Then, it is clear from the assumption that $\lambda+\mu\le 6r-9$ and
    $$6r=\sum_{x\in\cup_{i=1}^6 A_i}\deg(x)=2\lambda+\mu\le 2\lambda+(6r-9-\lambda)=6r-9+\lambda,$$

    \noindent which implies that $\lambda\ge9$. Theorem \ref{clafy6r-9} in fact shows that $|\cup_{i=1}^6 A_i|=6r-9$ and the subhypergraph $\{A_1,\ldots,A_6\}$ contains precisely nine vertices of degree two and $6r-18$ vertices of degree one.
    In particular, the nine degree two vertices have only one possible configuration (note that we do not care about the degree one vertices since they will not appear in any intersections), which is described in the conclusion of the theorem above.
 \end{remark}

 \begin{proof}
    By the remark above we know that $\ma{H}$ is $\ma{G}_r(3r-3,3)$-free and $\ma{G}_r(5r-7,5)$-free. Let us apply Lemma \ref{e-1toe}, with, say, $A_6$ playing the role of $A_i$, and $A_1,A_2,A_3$ playing the roles of $A_{i_1},A_{i_2},A_{i_3}$, respectively. Assume without loss of generality that $A_1\cap A_6=\{a\}\in V_1$, $A_2\cap A_6=\{b\}\in V_2$, and $A_3\cap A_6=\{c\}\in V_3$, as depicted below.

    \begin{center}
    \begin{tabular}{|c|c|c|c|c|c|c|}
      \hline
       & $A_1$ & $A_2$ & $A_3$ & $A_4$ & $A_5$ & $A_6$\\\hline
      $V_1$ & $a$ &  &  &  &  & $a$ \\\hline
      $V_2$ &  & $b$ &  &  &  & $b$\\\hline
      $V_3$ &  &  & $c$ &  &  & $c$\\\hline
    \end{tabular}
        \end{center}

    We will make use of the following two claims, whose proofs are postponed to the end of this subsection.

    \begin{claim}\label{claim-1}
    $A_1,A_2,A_3$ are pairwise disjoint and so are $A_4,A_5,A_6$. Moreover, $$|(A_1\cup A_2\cup A_3)\cap (A_4\cup A_5\cup A_6)|=9$$ and $\{A_i\cap A_j:1\le i\le 3,~4\le j\le 6\}$ is a set of nine distinct vertices.
    \end{claim}

        \begin{claim}\label{claim-2}
        The six vertices which appear in the intersections $A_4\cap(A_1\cup A_2\cup A_3)$ and $A_5\cap(A_1\cup A_2\cup A_3)$ are all contained in the vertex parts $V_1,V_2,V_3$, where we have assumed that $A_1\cap A_6=\{a\}\in V_1$, $A_2\cap A_6=\{b\}\in V_2$ and $A_3\cap A_6=\{c\}\in V_3$.
    \end{claim}

Assuming the correctness of the two claims above, let us continue the proof of the theorem.
    On one hand, Claim \ref{claim-1} implies that $|\cup_{i=1}^6 A_i|=6r-9$ and proves conclusions (1), (2) of the theorem. On the other hand, it is easy to see that Claim \ref{claim-2} proves conclusion (3) of the theorem. Moreover, as $A_1\setminus(V_1\cup V_2\cup V_3),\ldots,A_6\setminus(V_1\cup V_2\cup V_3)$ are pairwise disjoint, for the sake of classification there is no need to consider the vertices that are contained in those subsets. Assume without loss of generality that the restrictions of $A_1,A_2,A_3$ to $V_1,V_2,V_3$ are given by the following table.


    \begin{center}
    \begin{tabular}{|c|c|c|c|}
      \hline
       & $A_1$ & $A_2$ & $A_3$ \\\hline
      $V_1$  & $a$  & $f$  & $h$\\\hline
      $V_2$ & $d$  & $b$ & $i$\\\hline
      $V_3$ & $e$  & $g$ & $c$ \\\hline
    \end{tabular}
        \end{center}

    By the linearity of $\ma{H}$ and the two claims above, it is easy to check by case analysis that up to isomorphism, there is only one possibility for the choice of $A_4$ and $A_5$, when considering their restrictions to $V_1\cup V_2\cup V_3$, which is exactly the one described by Table \ref{tab63not96}.  Indeed, given that $|(A_1\cup A_2\cup A_3)\cap (A_4\cup A_5\cup A_6)|=9$, and that all nines vertices in the intersection must be contained in $V_1\cup V_2\cup V_3$, all that remains is to choose the vertex classes of the nine vertices, which is the same task as edge-coloring a $K_{3,3}$. It is well known that there is a unique such edge-coloring up to isomorphism.

   The proof of the theorem is thus completed.
  \end{proof}

  It remains to prove Claims \ref{claim-1} and \ref{claim-2}.

  \begin{proof}[\textbf{Proof of Claim \ref{claim-1}}] First of all, it follows by  Lemma \ref{e-1toe} that $A_1,A_2$ and $A_3$ are pairwise disjoint.
    Then, by applying Lemma \ref{e-1toe} to $A_i,~1\le i\le 3$, we conclude that there exist (not necessarily distinct) edges  $A_{ij}\in\{A_1,\ldots,A_6\}$, where $1\le i,j\le 3$, such that
    \begin{itemize}
        \item $|A_1\cap(A_{11}\cup A_{12}\cup A_{13})|=3,~|A_2\cap(A_{21}\cup A_{22}\cup A_{23})|=3,~|A_3\cap(A_{31}\cup A_{32}\cup A_{33})|=3,$
        \item for each $i$, $A_{i1}, A_{i2}, A_{i3}$ are pairwise disjoint.
    \end{itemize}

     \noindent 
     By the pairwise disjointness of $A_1,A_2,A_3$, it is clear that for $1\le i\le 3$, we have $\{A_{i1}, A_{i2}, A_{i3}\}=\{A_4,A_5,A_6\}$, which implies that $A_4,A_5,A_6$ are also pairwise disjoint. Moreover, by 
     the discussion above, it is not hard to check that the nine vertices contained in the intersections $A_i\cap (A_4\cup A_5\cup A_6)$, $1\le i\le 3$, are all distinct. The proof of the claim is thus completed.
    \end{proof}

        \begin{proof}[\textbf{Proof of Claim \ref{claim-2}}]
        By symmetry, it suffices to show that $A_4\cap A_1\in V_1\cup V_2\cup V_3$.
    Note that the claim holds fairly straightforwardly for 3-partite hypergraphs. For $r\ge 4$, suppose that there exists some $d$ such that $A_4\cap A_1=\{d\}\not\in V_1\cup V_2\cup V_3$. For simplicity we assume that $d\in V_4$ for some vertex part $V_4\not\in\{V_1,V_2,V_3\}$.

    We claim that in this case $A_4\cap A_2$ and $A_4\cap A_3$ must belong to $V_1\cup V_2\cup V_3$. By Claim \ref{claim-1}, it is easy to see that $A_4\cap A_2\neq \{d\}$ and $A_4\cap A_3\neq \{d\}$. Therefore, by the $r$-partiteness of $\ma{H}$, it is clear that neither of $A_4\cap A_2$ and $A_4\cap A_3$ is contained in $V_4$. Thus the statement above holds trivially for 4-partite hypergraphs.
    For $r\ge 5$, suppose that $A_4\cap A_2=\{e\}$ and $e\in V_5$ for some vertex part $V_5\not\in\{V_1,V_2,V_3,V_4\}$.

    According to the discussion above, we can characterize the degree two vertices in $A_1\cup A_2\cup A_3\cup A_6$ by the following Table \ref{tabv4v5}. It is not hard to check that $$d,A_1,a,A_6,b,A_2,e,A_4,d$$ form a rainbow 4-cycle, a contradiction.
    Analogously, one can show that $A_4\cap A_3\in V_1\cap V_2\cap V_3$, completing the proof of the claim.
    \begin{table}[ht]
    \begin{center}
    \begin{tabular}{|c|c|c|c|c|c|}
      \hline
       & \bm{$A_1$} & \bm{$A_2$} & $A_3$ & \bm{$A_4$}  & \bm{$A_6$}\\\hline
      $V_1$ & $a$ &  &  &  &   $a$ \\\hline
      $V_2$ &  & $b$ &  &  &   $b$\\\hline
      $V_3$ &  &  & $c$ &  &   $c$\\\hline
      $V_4$ & $d$  &  &  & $d$   & \\\hline
      $V_5$ & & $e$ & &$e$ &\\
      \hline
    \end{tabular}
        \end{center}
         \caption{$A_4\cap A_1\in V_4$ and $A_4\cap A_2\in V_5$, the bold edges form a rainbow 4-cycle}\label{tabv4v5}
         \end{table}


        Recall that $A_4\cap\{a,b,c\}=\emptyset$ and that $A_2\cap V_2=\{b\}$. Then, $A_4\cap A_2$ is located in either $V_1$ or $V_3$. On one hand, if $A_4\cap A_2=\{e\}\in V_3$, then Table \ref{tabv4v3} below indicates that $$d,A_1,a,A_6,b,A_2,e,A_4,d$$ must form a rainbow 4-cycle and we arrive at a contradiction.

        \begin{table}[h]
        \begin{center}
        \begin{tabular}{|c|c|c|c|c|c|}
      \hline
      & \bm{$A_1$} & \bm{$A_2$} & $A_3$ & \bm{$A_4$}  & \bm{$A_6$}\\\hline
      $V_1$ & $a$ &  &  &  &   $a$ \\\hline
      $V_2$ &  & $b$ &  &  &   $b$\\\hline
      $V_3$ &  & $e$  & $c$ & $e$ &   $c$\\\hline
      $V_4$ & $d$  &  &  & $d$   & \\\hline
    \end{tabular}
        \end{center}
        \caption{$A_4\cap A_1\in V_4$ and $A_4\cap A_2\in V_3$, the bold edges form a rainbow 4-cycle}\label{tabv4v3}
        \end{table}

 \noindent On the other hand, if $A_4\cap A_2=\{e\}\in V_1$, then $A_4\cap A_3$ must belong to $V_2$ (it cannot belong to $V_1$ as $A_4\cap A_2\in V_1$, and it cannot belong to $V_3$ as $A_3\cap V_3=\{c\}$ and $c\not\in A_4$). Let us denote $A_4\cap A_3=\{f\}\in V_2$. Then Table \ref{tabv4v1} below indicates that $$d,A_1,a,A_6,c,A_3,f,A_4,d$$ again form a rainbow 4-cycle, which is also a contradiction.

    \begin{table}[ht]
        \begin{center}
        \begin{tabular}{|c|c|c|c|c|c|}
      \hline
       & \bm{$A_1$} & $A_2$ & \bm{$A_3$} & \bm{$A_4$} &  \bm{$A_6$}\\\hline
      $V_1$ & $a$ & $e$ &  & $e$ &   $a$ \\\hline
      $V_2$ &  & $b$ & $f$ & $f$ &   $b$\\\hline
      $V_3$ &  &   & $c$ &  &   $c$\\\hline
      $V_4$ & $d$  &  &  & $d$ &   \\\hline
    \end{tabular}
        \end{center}
        \caption{$A_4\cap A_1\in V_4$, $A_4\cap A_2\in V_1$, $A_3\cap A_4\in V_2$, the bold edges form a rainbow 4-cycle}\label{tabv4v1}
    \end{table}

      Finally, we conclude that all six vertices which are contained in the intersections $A_4\cap(A_1\cup A_2\cup A_3)$ and $A_5\cap(A_1\cup A_2\cup A_3)$ must be located in $V_1,V_2$ and $V_3$, completing the proof of Claim \ref{claim-2}.
        \end{proof}



%
%


  \subsection{$\ma{G}_r(7r-11,7)$-free hypergraphs}

\noindent Our goal in this subsection is to prove the following theorem.

 \begin{theorem}\label{7r-11}
    Let $r\ge 3$ be a positive integer and $\ma{H}$ be a linear $r$-partite $r$-graph. Assume that $\ma{H}$ contains no rainbow cycles of length three or four. Then $\ma{H}$ is $\ma{G}_r(7r-11,7)$-free.
\end{theorem}

To prove the theorem, we will need the following two lemmas.

  \begin{lemma}\label{int6r-9<2}
    Let $r\ge 3$ be a positive integer and $\ma{H}$ be a linear $r$-partite $r$-graph. Assume that $\ma{H}$ contains no rainbow cycles of length three or four. If there exist six edges $A_1,\ldots,A_6$ of $\ma{H}$ with $|\cup_{i=1}^6 A_i|\le 6r-9$, then for any other edge $A_7\in\ma{H}\setminus\{A_1,\ldots,A_6\}$, it holds that $|A_7\cap(\cup_{i=1}^6 A_i)|\le 1$.
\end{lemma}

  \begin{proof}
  Recall that by Theorem \ref{clafy6r-9}, $A_1,\ldots,A_6$ have only one possible configuration. 
  Denote $X=\cup_{i=1}^6 A_i$. Assume towards contradiction that there exists some $A_7\in\ma{H}\setminus\{A_1,\ldots,A_6\}$ with $|A_7\cap X|\ge 2$. Then, by the linearity of $\ma{H}$ there exist distinct $i,j\in[6]$ such that $|A_7\cap A_i|=|A_7\cap A_j|=1$ and $A_7\cap A_i\neq A_7\cap A_j$. Therefore, it follows that $A_i\cap A_j=\emptyset$, since otherwise $A_7,A_i$ and $A_j$ will violate the $\ma{G}_r(3r-3,3)$-free property of $\ma{H}$. Consequently, by Theorem \ref{clafy6r-9} (1) we have either $\{i,j\}\s\{1,2,3\}$ or $\{i,j\}\s\{4,5,6\}$.

  Assume without loss of generality that $\{i,j\}\s\{1,2,3\}$. Denote $\{y_i\}=A_7\cap A_i$ and $\{y_j\}=A_7\cap A_j$. Let $V_1,V_2,V_3$ be the vertex parts of $\ma{H}$, given by the proof of Theorem \ref{clafy6r-9}. The remainder of the proof is divided into two cases.



    \vspace{10pt}

    \noindent\textbf{Case 1.} $\{y_i,y_j\}\cap (V_1\cup V_2\cup V_3)\neq\emptyset$.

    \vspace{10pt}

    By symmetry, it is sufficient to consider the case $y_j\in V_1\cup V_2\cup V_3$. Then, by Theorem \ref{clafy6r-9} (2) there exists a $j'\in\{4,5,6\}$ with $A_{j'}\cap A_j=\{y_j\}$ and $|A_{j'}\cap A_i|=1$. As $A_i,A_{j'},A_7$ are pairwise intersecting, and $A_{j'}\cap A_7\neq A_i\cap A_7$, it is not hard to check that $|A_i\cup A_{j'}\cup A_7|\le 3r-3$, which is a contradiction.

    \vspace{10pt}

    \noindent\textbf{Case 2.} $\{y_i,y_j\}\cap (V_1\cup V_2\cup V_3)=\emptyset$.

     \vspace{10pt}

      In this case, by the $r$-partiteness of $\ma{H}$ it is clear that $y_i$ and $y_j$ cannot be contained in the same vertex part. Assume that $y_i\in V_{l_i}$ and $y_j\in V_{l_j}$, where $\{l_i,l_j\}\cap\{1,2,3\}=\emptyset$ and $l_i\neq l_j$. 
      According to Theorem \ref{clafy6r-9} (2), there exists a  $k\in\{4,5,6\}$ such that $A_k\cap A_i=\{x_i\},~A_k\cap A_j=\{x_j\}$, and $\{x_i,x_j\}\s V_1\cup V_2\cup V_3$. It is not hard to check that $$x_i,A_i,y_i,A_7,y_j,A_j,x_j,A_k,x_i$$ form a rainbow 4-cycle, and we arrive at a contradiction.
      The proof of the lemma is thus completed.
    \end{proof}

%

\begin{lemma}\label{inter4atmost1}
    Let $r\ge 4$ be a positive integer and $\ma{H}$ be a linear $r$-partite $r$-graph. Assume that $\ma{H}$ contains no rainbow cycles of length three or four, and that it contains no vertex with degree larger than two. Let $A_1,A_2,A_3,A_4$ be four pairwise disjoint edges of $\ma{H}$. Then, there exists at most one edge $B\in\ma{H}\setminus\{A_1,A_2,A_3,A_4\}$ with $|B\cap(\cup_{i=1}^4 A_i)|=4$.
\end{lemma}

\begin{proof}
 By the linearity of $\ma{H}$, it is clear that for any $B\in\ma{H}\setminus\{A_1,A_2,A_3,A_4\}$ we have $|B\cap(\cup_{i=1}^4 A_i)|\le 4$. Assume towards contradiction that there exist two distinct edges $B,C\in\ma{H}\setminus\{A_1,A_2,A_3,A_4\}$ with $|B\cap(\cup_{i=1}^4 A_i)|=|C\cap(\cup_{i=1}^4 A_i)|=4$. It is easy to see that the eight vertices contained in the two intersections above are pairwise distinct, since $\ma{H}$ contains no vertex with degree larger than two. Then, by the $\ma{G}_r(3r-3,3)$-free property of $\ma{H}$ it is not hard to verify that $B\cap C=\emptyset$. In order to obtain the desired contradiction, we proceed to show that the subhypergraph $\{A_1,A_2,A_3,A_4,B,C\}$ must contain a rainbow $4$-cycle.

 To this end, let $V_1,\ldots,V_r$ be the $r$ vertex parts of $\ma{H}$. Assume without loss of generality that for $1\le i\le 4$, $B\cap A_i=\{b_i\}\in V_i$ and $C\cap A_i=\{c_i\}$. Since $b_i\neq c_i$ and $\ma{H}$ is $r$-partite, it is clear that $c_i\not\in V_i$.
 We claim that for any $1\le i_1<i_2\le 4$, it holds that $\{c_{i_1},c_{i_2}\}\cap(V_{i_1}\cup V_{i_2})\neq\emptyset$. Suppose the opposite, then the vertices $b_{i_1},b_{i_2},c_{i_1},c_{i_2}$ belong to four distinct vertex parts, and
 $$b_{i_1},A_{i_1},c_{i_1},C,c_{i_2},A_{i_2},b_{i_2},B,b_{i_1}$$ clearly form a rainbow 4-cycle, which is a contradiction.

 Assume that $c_1\in V_i$ for some $i\neq 1$. It is clear that there exist at least two distinct $i_1,i_2\in\{2,3,4\}\setminus\{i\}$ such that $c_1\not\in V_{i_1}$ and $c_1\not\in V_{i_2}$. Then, by the claim above, we have that
 $$\{c_1,c_{i_1}\}\cap (V_1\cup V_{i_1})\neq\emptyset\text{ and }\{c_1,c_{i_2}\}\cap (V_1\cup V_{i_2})\neq\emptyset,$$

 \noindent which implies that $\{c_{i_1},c_{i_2}\}\in V_1$. However, it is impossible since by definition we also have $\{c_{i_1}, c_{i_2}\}\s C$ and $|C\cap V_1|=1$. 
\end{proof}

Now we are in a position to present the proof of Theorem \ref{7r-11}. 


\begin{proof}[\textbf{Proof of Theorem \ref{7r-11}}]
    Assume toward contradiction that $\ma{H}$ is not $\ma{G}_r(7r-11,7)$-free. Then, there exists a subhypergraph $\ma{H}'=\{A_1,\ldots,A_7\}\s\ma{H}$ such that $|\cup_{i=1}^7 A_i|\le 7r-11$. 
    The proof is divided into two parts, according to whether $\ma{H}'$ is $\ma{G}_r(6r-9,6)$-free.

    \vspace{10pt}

    \noindent\textbf{Case 1.} $\ma{H}'$ is not $\ma{G}_r(6r-9,6)$-free.

    \vspace{10pt}


    Suppose that $A_1,\ldots,A_6$ are the six edges whose union contains at most $6r-9$ vertices. Denote $X=\cup_{i=1}^6 A_i$. It follows from Theorem \ref{clafy6r-9} that $|X|=6r-9$. 
    Consequently, 
    $$7r-11\ge|X\cup A_7|=|X|+|A_7|-|X\cap A_7|=6r-9+r-|X\cap A_7|,$$

    \noindent which implies that $|X\cap A_7|\ge 2$, contradicting the conclusion of Lemma \ref{int6r-9<2}.

     \vspace{10pt}

    \noindent\textbf{Case 2.} $\ma{H}'$ is $\ma{G}_r(6r-9,6)$-free.

    \vspace{10pt}

    Lemma \ref{maxdeg} then indicates that $\ma{H}'$ contains no vertex of degree larger than two. Assume that $\ma{H}'$ contains $\lambda$ degree two vertices and $\mu$ degree one vertices. It is clear that $\lambda+\mu\le 7r-11$.  Moreover, 
    $$7r=\sum_{x\in V(\ma{H}')}\deg(x)=2\lambda+\mu\le 2\lambda+(7r-11-\lambda)=7r-11+\lambda,$$

    \noindent which implies that $\lambda\ge 11$. Let us count the number of pairs
    $$N:=\left\{(v,A):v\in A,~A\in\ma{H}',~\deg(v)=2\right\}.$$

    \noindent Since $N=2\lambda\ge 22$ and $|\ma{H}'|=7$, there exists at least one edge of $\ma{H}'$ that contains at least four degree two vertices (as these four vertices must be located in four distinct vertex parts, the theorem holds trivially for $r=3$). Assume that $A_7$ is the edge that contains at least four degree two vertices of $V(\ma{H}')$. Let $A_1,A_2,A_3,A_4$ be the four edges each of which shares a degree two vertex with $A_7$.
   For $1\le i\le 4$, let $A_7\cap A_i=\{a_i\}$.

       Using the argument which proves Lemma \ref{e-1toe} (2), one can infer that $A_1,A_2,A_3,A_4$ are pairwise disjoint. Let us apply Lemma \ref{e-1toe} separately to $A_1,A_2,A_3,A_4$. For each $1\le i\le 4$, there exist three pairwise disjoint edges $A_{i_1},A_{i_2},A_{i_3}\in\{A_1,\ldots,A_7\}\setminus\{A_i\}$ such that $|A_{i_l}\cap A_i|=1$ for every $1\le l\le 3$. It follows fairly straightforwardly from the pairwise disjointness of $A_1,A_2,A_3,A_4$ that we must have  $\{A_{i_1},A_{i_2},A_{i_3}\}=\{A_5,A_6,A_7\}$ for all $1\le i\le 4$. Moreover, it is also not hard to see that the twelve vertices contained in the intersections $$A_5\cap(\cup_{i=1}^4 A_i),~A_6\cap(\cup_{i=1}^4 A_i),~A_7\cap(\cup_{i=1}^4 A_i)$$

       \noindent are all distinct, which further implies that
       $$|A_5\cap(\cup_{i=1}^4 A_i)|=|A_6\cap(\cup_{i=1}^4 A_i)|= |A_7\cap(\cup_{i=1}^4 A_i)|=4.$$ Therefore, we arrive at a contradiction by Lemma \ref{inter4atmost1}, completing the proof of the theorem.
    \end{proof}

     \subsection{$\ma{G}_r(8r-13,8)$-free hypergraphs}

      \noindent In this subsection, we will prove the following theorem.

       \begin{theorem}\label{8r-13}
    Let $r\ge 3$ be a positive integer, and let $\ma{H}$ be a linear $r$-partite $r$-graph. Assume that $\ma{H}$ contains no rainbow cycles of length three or four. Then $\ma{H}$ is $\ma{G}_r(8r-13,8)$-free.
\end{theorem}

\begin{proof}
    Suppose $\ma{H}$ is not $\ma{G}_r(8r-13,8)$-free. Let $\ma{H}'=\{A_1,\ldots,A_8\}$ be a subhypergraph of $\ma{H}$ with at most $8r-13$ vertices. Observe that by Theorem \ref{7r-11}, $\ma{H}'$ is $\ma{G}_r(7r-11,7)$-free. Then, it follows by Lemma \ref{maxdeg} that $\ma{H}'$ contains no vertex of degree larger than two. Assume that $\ma{H}'$ contains $\lambda$ degree two vertices and $\mu$ degree one vertices. It is clear that $\lambda+\mu\le 8r-13$.  Moreover, 
    $$8r=\sum_{x\in V(\ma{H}')}\deg(x)=2\lambda+\mu\le 2\lambda+(8r-13-\lambda)=8r-13+\lambda,$$

    \noindent which implies that $\lambda\ge 13$. Similar to the proof of Theorem \ref{7r-11}, by counting the number of pairs $$N:=\{(v,A):v\in A,~A\in\ma{H}',~\deg(v)=2\}$$
    one can show that there exists at least one edge of $\ma{H}'$ that contains at least four degree two vertices.

    Assume that $A_8$ is such an edge. Let $A_1,A_2,A_3,A_4$ be the four edges each of which shares a degree two vertex with $A_8$.
    Again, using the argument which proves Lemma \ref{e-1toe} (2), it is easy to verify that $A_1,A_2,A_3,A_4$ are pairwise disjoint. In the sequel we denote for simplicity that  $X=\cup_{i=1}^4 A_i$ and $Y=\cup_{j=5}^8 A_j$.

    We have the following claim, whose proof is postponed to the end of this subsection.

    \begin{claim}\label{claim-3}
    $A_5,A_6,A_7,A_8$ are pairwise disjoint.
    \end{claim}

    Assuming the correctness of Claim \ref{claim-3}, for each $j\in\{5,6,7\}$, by applying Lemma \ref{e-1toe} to $A_j$ it is not hard to see that there exist at least three distinct $A_{j_1},A_{j_2},A_{j_3}\in\{A_1,A_2,A_3,A_4\}$ such that $A_{j_l}\cap A_j\neq\emptyset$ for every $1\le l\le 3$. On the other hand, since $A_1,A_2,A_3,A_4$ are pairwise disjoint and $|A_8\cap X|=4$, it follows by Lemma \ref{inter4atmost1} that at most three of $\{A_1,A_2,A_3,A_4\}$ can have nonempty intersection with $A_j$. More precisely, we have 

    \begin{equation}\label{eq:new-1}
      |A_8\cap X|=4\text{ and }|A_j\cap X|=3\text{ for $j\in\{5,6,7\}$.}
    \end{equation}

    \noindent 
    Moreover, since $\ma{H}'$ contains no vertex of degree larger than two, the $4+3\times 3=13$ vertices contained in the intersections $A_j\cap X$, $5\le j\le 8$, are all distinct, which further implies that
    $$|X\cap Y|=13.$$

    Since $A_5,A_6,A_7,A_8$ are pairwise disjoint, by Lemma \ref{inter4atmost1} there exists at most one $i\in[4]$ such that $|A_i\cap Y|=4$. Therefore, there exists precisely one $i\in[4]$ such that $|A_i\cap Y|=4$, and for every $i'\in[4]\setminus\{i\}$, $|A_{i'}\cap Y|=3$.

    For notational convenience, we assume that $|A_1\cap Y|=4$.
    According to \eqref{eq:new-1}, for each $j\in\{5,6,7\}$ there exists a subset $I_j\s[4]$ with $|I_j|=3$ such that $|A_j\cap A_i|=1$ for all $i\in I_j$. It is not too difficult to check that $$\{I_5,I_6,I_7\}=\big\{\{1,2,3\},\{1,2,4\},\{1,3,4\}\big\},$$

    \noindent since otherwise one would find an $i'\in[4]\setminus\{1\}$ such that $|A_i\cap Y|=4$, which is a contradiction.

    By symmetry, we may assume that $I_5=\{1,2,3\}$, $I_6=\{1,2,4\}$ and $I_7=\{1,3,4\}$. For $1\le i\le 4$, assume without loss of generality that $A_8\cap A_i=\{a_i\}\in V_i$. Moreover, let us denote for $i\in I_5$, $A_i\cap A_5=\{b_i\}$, for $i\in I_6$, $A_i\cap A_6=\{c_i\}$, and for $i\in I_7$, $A_i\cap A_7=\{d_i\}$.

    To conclude the proof, we proceed to show that $\ma{H}'$ has to contain a rainbow 4-cycle.
    Note that the argument presented below is in some sense 
    in the spirit of Lemma \ref{inter4atmost1}.

    We will make use of the following two claims, whose proofs are easy to present.

    \begin{claim}\label{claim-4}
    For $i_1,i_2\in I_5$ (resp. $I_6$ and $I_7$) and $i_1\neq i_2$, it holds that $\{b_{i_1},b_{i_2}\}\cap(V_{i_1}\cup V_{i_2})\neq\emptyset$ (resp. $\{c_{i_1},c_{i_2}\}\cap(V_{i_1}\cup V_{i_2})\neq\emptyset$ and $\{d_{i_1},d_{i_2}\}\cap(V_{i_1}\cup V_{i_2})\neq\emptyset$).
    \end{claim}

    \vspace{10pt}

    \noindent{\it Proof.} By symmetry, it is sufficient to prove the claim for $I_5$. Assume towards contradiction that $\{b_{i_1},b_{i_2}\}\cap(V_{i_1}\cup V_{i_2})\neq\emptyset$, then
 $$b_{i_1},A_{i_1},a_{i_1},A_8,a_{i_2},A_{i_2},b_{i_2},A_5,b_{i_1}$$
 clearly form a rainbow 4-cycle, which is a contradiction.

   \vspace{10pt}

  \begin{claim}\label{claim-5}
$\{b_1,b_2,b_3\}\s V_1\cup V_2\cup V_3$, $\{c_1,c_2,c_4\}\s V_1\cup V_2\cup V_4$, and $\{d_1,d_3,d_4\}\s V_1\cup V_3\cup V_4$.
    \end{claim}

    \vspace{10pt}

    \noindent{\it Proof.} We will only prove the claim for $\{b_1,b_2,b_3\}$. Assume towards contradiction that $\{b_1,b_2,b_3\}\nsubseteq \cup_{i=1}^3 V_i$, say,
    $b_1\not\in\cup_{i=1}^3 V_i$. Then, by Claim \ref{claim-4} we have that $$\{b_1,b_2\}\cap(V_1\cup V_2)\neq\emptyset\text{ and }\{b_1,b_3\}\cap(V_1\cup V_3)\neq\emptyset.$$
    Since $A_2\cap V_2=\{a_2\},~A_3\cap V_3=\{a_3\}$, by the $r$-partiteness of $\ma{H}$ it is easy to see that $b_2\not\in V_2$ and $b_3\not\in V_3$. Therefore, the last two inequalities hold if and only if $\{b_2,b_3\}\s V_1$, which is impossible. 

    \vspace{10pt}

    Claim \ref{claim-5} implies that $X\cap Y\s \cup_{i=1}^4 V_i$. Next, we will choose the vertex classes for the vertices contained in that intersection. Recall that for each $i\in[4]$ we have $a_i\in V_i$. According to $r$-partiteness of $\ma{H}$ and the definitions of $b_i,c_j,d_k$, it is not hard to see that for all
    $i,j,k\in[4]$, we have  $b_i\not\in V_i$, $c_j\in V_j$ and $d_k\not\in V_k$.
    Note that by assumption we have $A_1\cap Y=A_1\cap (\cup_{i=1}^4 V_i)=\{a_1,b_1,c_1,d_1\}$.
    Moreover,  by Claim \ref{claim-5} we have $b_1\not\in V_4,~c_1\not\in V_3$, and $d_1\not\in V_2$.
    Then, it is not hard to check that either $b_1\in V_2,~c_1\in V_4,~d_1\in V_3$, or $b_1\in V_3,~c_1\in V_2,~d_1\in V_4$.

    For the first case, one can draw the following Table \ref{tab4deg2inonee=8}. Using Claims \ref{claim-4} and \ref{claim-5}, below we will show that given the vertex classes of $b_1,c_1,d_1$,  the vertex classes of  $b_2,b_3,c_2,c_4,d_3,d_4$ are also determined. 

    \begin{table}[ht]
      \begin{center}
       \begin{tabular}{|c|c|c|c|c|c|c|c|c|}
      \hline
       & $A_1$ & $A_2$ & $A_3$ & $A_4$ & $A_5$ & $A_6$ & $A_7$ & $A_8$ \\\hline
      $V_1$ & $a_1$ &       &       &       &       &       &       & $a_1$ \\\hline
      $V_2$ & $b_1$ & $a_2$ &       &       & $b_1$ &       &       & $a_2$  \\\hline
      $V_3$ & $d_1$ &       & $a_3$ &       &       & $d_1$ &       & $a_3$\\\hline
      $V_4$ & $c_1$ &       &       & $a_4$ &       &       & $c_1$ & $a_4$\\\hline
    \end{tabular}
        \end{center}
        \caption{The first case, $b_1\in V_2,~c_1\in V_4,~d_1\in V_3$}\label{tab4deg2inonee=8}
    \end{table}

    \noindent Indeed, 
    \begin{itemize}
        \item using $\{b_2,b_3\}\cap (V_2\cup V_3)\neq\emptyset$ and $b_1\in V_2$ one can infer that $b_2\in V_3$; moreover, using $\{b_1,b_3\}\cap (V_1\cup V_3)\neq\emptyset$, one can infer that $b_3\in V_1$;

        \item similarly, using $\{c_2,c_4\}\cap (V_2\cup V_4)\neq\emptyset$ and $c_1\in V_4$ one can infer that $c_4\in V_2$; moreover, using $\{c_1,c_2\}\cap (V_1\cup V_2)\neq\emptyset$, one can infer that $c_2\in V_1$;

        \item lastly, using $\{d_3,d_4\}\cap (V_3\cup V_4)\neq\emptyset$ and $d_1\in V_3$ one can infer that $d_3\in V_4$; moreover, using $\{d_1,d_4\}\cap (V_1\cup V_4)\neq\emptyset$, one can infer that $d_4\in V_1$.
    \end{itemize}

    To sum up, we can complete Table \ref{tab4deg2inonee=8} into the following Table \ref{table-new-1}. It is easy to see that the hypergraph described by Table \ref{table-new-1} contains a rich structure of rainbow 4-cycles. For example,
    $$b_1,A_1,c_1,A_6,c_2,A_2,b_2,A_5,b_1$$
    form a rainbow 4-cycle, which is a contradiction.

    \begin{table}[h]
      \begin{center}
       \begin{tabular}{|c|c|c|c|c|c|c|c|c|}
      \hline
            & $A_1$ & $A_2$ & $A_3$ & $A_4$ & $A_5$ & $A_6$ & $A_7$ & $A_8$ \\\hline
      $V_1$ & $a_1$ & $c_2$ & $b_3$ & $d_4$ & $b_3$ & $c_2$ & $d_4$ & $a_1$ \\\hline
      $V_2$ & $b_1$ & $a_2$ &       & $c_4$ & $b_1$ & $c_4$ &       & $a_2$  \\\hline
      $V_3$ & $d_1$ & $b_2$ & $a_3$ &       & $b_2$ &       & $d_1$ & $a_3$\\\hline
      $V_4$ & $c_1$ &       & $d_3$ & $a_4$ &       & $c_1$ & $d_3$ & $a_4$\\\hline
    \end{tabular}
        \end{center}
        \caption{A completion of Table \ref{tab4deg2inonee=8}}\label{table-new-1}
    \end{table}

    The second case can be proved analogously. We omit its proof for the sake of saving space. We conclude that if $\ma{H}'$ is not $\ma{G}_r(8r-13,8)$-free, then it must contain a rainbow 4-cycle, completing the proof of the theorem.
\end{proof}

It remains to prove Claim \ref{claim-3}.

\begin{proof}[\textbf{Proof of Claim \ref{claim-3}}]


    We first show that $A_5,A_6$ and $A_7$ are pairwise disjoint.
    Assume towards contradiction that at least one pair of $A_5,A_6,A_7$ is intersecting, say, $A_6\cap A_7\neq\emptyset$.  Let us apply Lemma \ref{e-1toe} separately to $A_1,A_2,A_3,A_4$. Then, for each $1\le i\le 4$, there exist three pairwise disjoint edges $A_{i_1},A_{i_2},A_{i_3}\in\{A_1,\ldots,A_8\}\setminus\{A_i\}$ such that $|A_{i_l}\cap A_i|=1$ for every $1\le l\le 3$.
    Due to the disjointness of $A_1,A_2,A_3,A_4$, for each $1\le i\le 4$, the edges $A_{i_1},A_{i_2},A_{i_3}$ must be chosen from $\{A_5,A_6,A_7,A_8\}$. Hence, it is clear that for each $i$, at least two of $A_5,A_6,A_7$ must have nonempty and distinct intersections with $A_i$.

    For notational convenience, let us call such two edges an intersecting pair of $A_i$. By the $\ma{G}_r(3r-3,3)$-free property of $\ma{H}$, it is easy to see that the two edges that form an intersecting pair of some $A_i$ must be disjoint.
    Therefore, given the assumption that $A_6\cap A_7\neq\emptyset$, the intersecting pair for each $A_i$ can only be either $(A_5,A_6)$ or $(A_5,A_7)$. Since $A_5$ appears in both choices, it follows that $A_5\cap A_i\neq\emptyset$ for each $1\le i\le 4$. Consequently, we have that
    $$|A_5\cap X|=|A_8\cap X|=4,$$
    and we arrive at a contradiction by Lemma \ref{inter4atmost1}. Therefore, $A_5,A_6,A_7$ must be pairwise disjoint, as needed. 


    To prove the claim, it remains to show that for each $j\in\{5,6,7\}$, it holds that $A_j\cap A_8=\emptyset$.
    By symmetry, it is sufficient to show $A_5\cap A_8=\emptyset$. Assume towards contradiction that $A_5\cap A_8\neq\emptyset$. Applying Lemma \ref{e-1toe} to $A_5$, it follows that there exist three pairwise disjoint edges $A_{j_1},A_{j_2},A_{j_3}\in\{A_1,\ldots,A_8\}\setminus\{A_5\}$ such that $|A_{j_l}\cap A_5|=1$ for every $1\le l\le 3$.
    Since $A_5,A_6,A_7$ are pairwise disjoint, clearly $A_{j_1},A_{j_2},A_{j_3}$ can only be chosen from $\{A_1,A_2,A_3,A_4,A_8\}$. So there is at least one $i\in[4]$ such that $A_i\cap A_5\neq\emptyset$ and $A_i\cap A_5\neq A_8\cap A_5$. We conclude that $A_i,A_5,A_8$ are pairwise intersecting and they do not share a common vertex. It follows that $|A_i\cup A_5\cup A_8|\le 3r-3$, violating the $\ma{G}_r(3r-3,3)$-free property of $\ma{H}$.

    The proof of the claim is thus completed.
\end{proof}

\subsection{Proof of Theorem \ref{main}}

\noindent It is clear that Theorem \ref{main} is a direct consequence of Theorems \ref{3r-3im4r-5}, \ref{3r-3im5r-7}, \ref{7r-11} and \ref{8r-13}.

 \section{Linear 3-graphs containing neither triangles nor grids}\label{new-section-1}

\noindent The goal of this section is to prove Theorem \ref{96free}. Let $N\s\mathbb{F}_7^n$ be a subset that has no nontrivial solution to the equation $m_1+2m_2=3m_3$. It was known by the result of \cite{LinWolf} that there exists such a subset $N$ with size $|N|=\Omega(7^{2n/3})$. Next, let us construct a 3-partite 3-graph $\ma{H}_N$ which satisfies the conclusion of Theorem \ref{96free}. The vertex set $V(\ma{H}_N)$ is the disjoint union of three pairwise disjoint sets $V_1,V_2,V_3$, where for each $1\le i\le 3$, $V_i$ is a copy of $\mathbb{F}_7^n$. The edge set is defined as
   \begin{equation*}
       \ma{H}_N=\{(y,y+m,y+3m):y\in\mathbb{F}_7^n,~m\in N\}\s V_1\times V_2\times V_3=(\mathbb{F}_7^n)^3.
   \end{equation*}

The following lemma is easy to prove.

\begin{lemma}\label{3-graph-lemma}
     $\ma{H}_N$ is linear and $\ma{G}_3(6,3)$-free.
\end{lemma}

\begin{proof}
     Similar to the proofs of Lemma \ref{linear} and Theorem \ref{sumfree-rainbow}, it is easy to see that $\ma{H}_N$ is linear and contains no rainbow 3-cycles, respectively. Then, it follows from Lemma \ref{noTim63} that $\ma{H}_N$ is also $\ma{G}_3(6,3)$-free, as needed.
\end{proof}

\begin{proof}[\textbf{Proof of Theorem \ref{96free}}]
We proceed to show that $\ma{H}_N$ is $\ma{G}_3(9,6)$-free (thus it is also $G_{3\times 3}$-free). Assume towards contradiction that there exist six distinct edges $A_1,A_2,A_3,A_4,A_5,A_6$ of $\ma{H}_N$, whose union contains at most nine vertices.
We can represent those edges by the following table,


        \begin{center}
    \begin{tabular}{|c|c|c|c|c|c|c|}
      \hline
                                     & $A_1$                           & $A_2$                           & $A_3$                           & $A_4$                             & $A_5$                           & $A_6$\\\hline
      $V_1$     & $y_1$     & $y_2$    & $y_3$    &  $y_4$     & $y_5$    & $y_6$ \\\hline
      $V_2$       & $y_1+  m_1$       & $y_2+  m_2$       & $y_3+  m_3$      &  $y_4+  m_4$       & $y_5+  m_5$       & $y_6+  m_6$\\\hline
      $V_3$ & $y_1+3m_1$  & $y_2+3m_2$ & $y_3+3m_3$ &  $y_4+3m_4$ & $y_5+3m_5$ & $y_6+3m_6$\\\hline
    \end{tabular}
        \end{center}

    \noindent where for $1\le i\le 6$, $y_i\in \mathbb{F}_7^n$ and $m_i\in N$, and for $i\neq j$, $(y_i,m_i)\neq (y_j,m_j)$. Suppose that $A_i,~1\le i\le 6$, satisfy the conclusions of Theorem \ref{clafy6r-9}. Then, the following nine equalities must hold simultaneously:

    \begin{equation*}
      \left.\begin{aligned}
&y_4=y_1, \quad\quad\quad\qquad\qquad y_5=y_2, \quad\qquad\qquad\quad\quad y_6=y_3,\\
&y_4+ m_4=y_2+m_2, \quad\quad y_5+ m_5=y_3+m_3, \quad\quad y_6+ m_6=y_1+m_1,\\
&y_4+3m_4=y_3+3m_3, \quad y_5+3m_5=y_1+3m_1, \quad y_6+3m_6=y_2+3m_2.\\
      \end{aligned}\right.
    \end{equation*}


    \noindent We can rearrange the six equations in the bottom two rows as follows:
        \begin{equation*}
      \left.\begin{aligned}
&y_1+ m_4=y_2+m_2, \quad\quad y_2+ m_5=y_3+m_3, \quad\quad y_3+ m_6=y_1+m_1,\\
&y_1+3m_4=y_3+3m_3, \quad y_2+3m_5=y_1+3m_1, \quad y_3+3m_6=y_2+3m_2.\\
      \end{aligned}\right.
    \end{equation*}

\noindent In order to eliminate the $y_i$'s, by adding separately the both hand sides of the first and the fifth, the second and the sixth, and the third and the fourth equations, we obtain the following three identities:

$$m_4+3m_5=m_2+3m_1,\quad  m_5+3m_6=m_3+3m_2,\quad m_6+3m_4=m_1+3m_3.$$



\noindent As we are working on the finite field $\mathbb{F}_7$, it is easy to check by the equalities above that $$m_3+2m_2=3m_1\quad\text{ and }\quad m_5+2m_4=3m_5.$$
It thus follows from the definition of $N$ that $m_1=m_2=m_3$ and $m_4=m_5=m_6$, which implies that $(y_i,m_i)=(y_j,m_j)$ for all $i\neq j$, a contradiction.

The proof of Theorem \ref{96free} is thus completed.
\end{proof}

As mentioned in the introduction, Ellenberg and Gijswijt \cite{ellenberg2016large} showed that $r(\mathbb{F}_7^n)<c^n$ for some positive constant $c<7$. So the hypergraph $\ma{H}_N$ can never be used to show that $f_3(n,9,6)>n^{2-o(1)}$ holds for sufficiently large $n$. However, it is possible to prove the conjectured lower bound by using other solution-free sets. With the notation in \cite{Furediconst}, let $r(n,\star)$ denote the maximum size of a subset $M\s[n]$ with no nontrivial solution to the equation $$2x+2y=3z+w.$$ Ruzsa \cite{ruzsa} showed that $\Omega(n^{0.5})=r(n,\star)=o(n)$ and asked whether $r(n,\star)>n^{1-o(1)}$. Indeed, if one could answer the above question affirmatively, then it would imply that $f_3(n,9,6)>n^{2-o(1)}$.

To see this, we can construct a 3-partite 3-graph $\ma{H}_M$ as follows. The vertex set $V(\ma{H}_M)$ is the disjoint union of three pairwise disjoint sets $V_1,V_2,V_3$, where for each $1\le i\le 3$, $V_i$ is a copy of $[n^{1+o(1)}]$. The edge set is defined as
   \begin{equation*}
       \ma{H}_M=\{(y,y+m,y+2m):y\in[n],~m\in M\}\s V_1\times V_2\times V_3.
   \end{equation*}
\noindent It is not hard to see that $\ma{H}_M$ is linear and $\ma{G}_3(6,3)$-free. Similar to the proof of Theorem \ref{96free}, using Theorem \ref{clafy6r-9} one can also show that $\ma{H}_M$ must be $\ma{G}_3(9,6)$-free,which implies that $f_3(n,9,6)>n\cdot r(n,\star)$.


\section{A general upper bound for $f_r(n,v,e)$}\label{new-section-2}

\noindent In this section we will present the proof of Theorem \ref{generalF}. To that end, we first prove the following  recursive inequality, as stated in Lemma \ref{recursive} below. 
Let $R=er-v$ and denote $T_r(n,R,e):=f_r(n,v,e)$. We find that it is more convenient to work with $T_r(n,R,e)$ rather than $f_r(n,v,e)$. 

\begin{lemma}\label{recursive}
    For any $l\in[r]$, it holds that
    \begin{equation*}
        \begin{aligned}
            T_r(n,R,e)\le T_r(n,R-l,e-1)+\binom{n}{l}/\binom{r}{l}.
        \end{aligned}
    \end{equation*}
\end{lemma}

\begin{proof}
    For any $r$-graph $\ma{H}$ on $n$ vertices, we claim that there exists a subhypergraph $\ma{F}_{\ma{H}}\s\ma{H}$ with cardinality at most $\binom{n}{l}/\binom{r}{l}$ such that for any $A\in\ma{H}\setminus\ma{F}_{\ma{H}}$, there exists $B\in\ma{F}_{\ma{H}}$ with $|B\cap A|\ge l$.

    Let $\ma{F}$ be a maximal subfamily of $\ma{H}$ with the property that for any distinct $A,B\in\ma{F}$, $|A\cap B|\le l-1$. Next, we show that it is sufficient to take $\ma{F}_{\ma{H}}:=\ma{F}$. Indeed, if there is some edge $A\in\ma{H}\setminus\ma{F}$ with $|A\cap B|\le l-1$ for every $B\in\ma{F}$, then $\ma{F}\cup\{A\}$ also has the required property,  contradicting the maximality of $\ma{F}$. Therefore, to prove the claim it remains to show that $|\ma{F}|\le \binom{n}{l}/\binom{r}{l}$, which follows easily from the observation that any pair of distinct edges of $\ma{F}$ do not share a common $l$-subset.


    Let $\ma{H}$ be a $\ma{G}_r(v,e)$-free $r$-graph, and let $v=er-R$. Denote $\ma{H}'=\ma{H}\setminus\ma{F}_{\ma{H}}$. According to the claim above, it is clear that $$|\ma{H}|=|\ma{H}'|+|\ma{F}_{\ma{H}}|\le|\ma{H}'|+\binom{n}{l}/\binom{r}{l}.$$

    \noindent To prove the lemma, it suffices to show that $\ma{H}'$ is $\ma{G}_r(v',e-1)$-free with $v'=(e-1)r-R+l$, as by definition $f_r(n,v',e-1)=T_r(n,R-l,e-1)$. Assume for the contradiction that there exist distinct edges $A_1,\ldots,A_{e-1}\in\ma{H}'$ with $|\cup_{i=1}^{e-1} A_i|\le (e-1)r-R+l$. Then, by the definition of $\ma{F}_{\ma{H}}$, there exists at least one edge $B\in\ma{F}_{\ma{H}}$ such that $|B\cap A_1|\ge l$. Consequently,

    \begin{equation*}
        \begin{aligned}
          |B\cup(\cup_{i=1}^{e-1}A_i)|=|B|+|\cup_{i=1}^{e-1}A_i|-|B\cap(\cup_{i=1}^{e-1}A_i)|\le r+(e-1)r-R=v,
        \end{aligned}
    \end{equation*}
    \noindent which contradicts the assumption that $\ma{H}$ is $G_r(v,e)$-free.
\end{proof}

%

Lemma \ref{recursive} has several simple consequences, as listed below.

\begin{proposition}\label{recursiveF}
For any $l\in[r]$, it holds that
    \begin{equation*}
        \begin{aligned}
            f_r(n,v,e)\le f_r(n,v-r+l,e-1)+\binom{n}{l}/\binom{r}{l}.
        \end{aligned}
    \end{equation*}
\end{proposition}

\begin{proof}
    Note that $T_r(n,R-l,e-1)=f_r(n,v-r+l,e-1)$. 
\end{proof}



\begin{proposition}\label{generalT}
    Assume that $R=p(e-1)+q$, where $1\le q\le e-1$. Then it holds that
    $$T_r(n,R,e)\le q\binom{n}{p+1}/\binom{r}{p+1}+(e-1-q)\binom{n}{p}/\binom{r}{p}.$$
\end{proposition}

\begin{proof}
    It is easy to verify that $R=q(p+1)+(e-1-q)p$. Then, we can apply Lemma \ref{recursive} repeatedly for $e-1$ times, in which $l$ is chosen to be $p+1$ for $q$ times and to be $p$ for $e-1-q$ times. The conclusion of the proposition then follows from the easy fact that $T_r(n,0,1)=0$.
\end{proof}

\begin{proof}[\textbf{Proof of Theorem \ref{generalF}}]
    The theorem follows easily from Proposition \ref{generalT}. 
\end{proof}

\section*{Acknowledgements}
\noindent The authors wish to express their gratitude to the two anonymous reviewers for their careful reading and many constructive comments which are very helpful to the improvement of this paper. In particular, they want to thank one reviewer for bringing \cite{nagle-rodl-schacht} into their attention, and pointing out that Theorem \ref{upperbd} can also be proved by a reduction to \eqref{added}.

G. Ge is supported by the National Natural Science Foundation of China under Grant No. 11971325, National Key Research and Development Program of China under Grant No. 2018YFA0704703, and Beijing Scholars Program.

C. Shangguan is supported by the project of Qilu Young Scholars of Shandong University.

{\small\bibliographystyle{plain}
\bibliography{sparse}}

\begin{thebibliography}{10}

\bibitem{ipp}
N.~Alon, E.~Fischer, and M.~Szegedy.
\newblock Parent-identifying codes.
\newblock {\em J. Combin. Theory Ser. A}, 95(2):349--359, 2001.

\bibitem{t=3}
N.~Alon and A.~Shapira.
\newblock On an extremal hypergraph problem of {B}rown, {E}rd{\H o}s and
  {S}\'os.
\newblock {\em Combinatorica}, 26(6):627--645, 2006.

\bibitem{Behrend}
F.~A. Behrend.
\newblock On sets of integers which contain no three terms in arithmetical
  progression.
\newblock {\em Proc. Nat. Acad. Sci. U. S. A.}, 32:331--332, 1946.

\bibitem{hypercycle2}
C.~Berge.
\newblock Hypergraphs.
\newblock In {\em Selected topics in graph theory, 3}, pages 189--206. Academic
  Press, San Diego, CA, 1988.

\bibitem{hypercycle1}
C.~Berge.
\newblock {\em Hypergraphs}, volume~45 of {\em North-Holland Mathematical
  Library}.
\newblock North-Holland Publishing Co., Amsterdam, 1989.
\newblock Combinatorics of finite sets, Translated from the French.

\bibitem{bes2}
W.~G. Brown, P.~Erd{\H{o}}s, and V.~T. S{\'o}s.
\newblock Some extremal problems on {$r$}-graphs.
\newblock In {\em New directions in the theory of graphs ({P}roc. {T}hird {A}nn
  {A}rbor {C}onf., {U}niv. {M}ichigan, {A}nn {A}rbor, {M}ich, 1971)}, pages
  53--63. Academic Press, New York, 1973.

\bibitem{graphremovallemmas}
D.~Conlon and J.~Fox.
\newblock Graph removal lemmas.
\newblock In {\em Surveys in combinatorics 2013}, volume 409 of {\em London
  Math. Soc. Lecture Note Ser.}, pages 1--49. Cambridge Univ. Press, Cambridge,
  2013.

\bibitem{ellenberg2016large}
J.~S. Ellenberg and D.~Gijswijt.
\newblock On large subsets of $\mathbb{F}_q^n$ with no three-term arithmetic
  progression.
\newblock {\em Ann. of Math.}, 185:339--443, 2017.

\bibitem{Erdos1964}
P.~Erd\H{o}s.
\newblock Extremal problems in graph theory.
\newblock In {\em Theory of {G}raphs and its {A}pplications ({P}roc. {S}ympos.
  {S}molenice, 1963)}, pages 29--36. Publ. House Czechoslovak Acad. Sci.,
  Prague, 1964.

\bibitem{Erdosr=2}
P.~Erd\H{o}s.
\newblock Problems and results in combinatorial analysis.
\newblock pages 3--17. Atti dei Convegni Lincei, No. 17, 1976.

\bibitem{erdos2}
P.~Erd{\H{o}}s, P.~Frankl, and V.~R{\"o}dl.
\newblock The asymptotic number of graphs not containing a fixed subgraph and a
  problem for hypergraphs having no exponent.
\newblock {\em Graphs Combin.}, 2(2):113--121, 1986.

\bibitem{franklhypergraphremoval}
P.~Frankl and V.~R\"odl.
\newblock Extremal problems on set systems.
\newblock {\em Random Structures Algorithms}, 20(2):131--164, 2002.

\bibitem{Furediconst}
Z.~F{\"u}redi and M.~Ruszink{\'o}.
\newblock Uniform hypergraphs containing no grids.
\newblock {\em Adv. Math.}, 240:302--324, 2013.

\bibitem{Glock2019}
S.~Glock.
\newblock Triple systems with no three triples spanning at most five points.
\newblock {\em Bull. Lond. Math. Soc.}, 51(2):230--236, 2019.

\bibitem{gowers1}
W.~T. Gowers.
\newblock Quasirandomness, counting and regularity for 3-uniform hypergraphs.
\newblock {\em Combin. Probab. Comput.}, 15(1-2):143--184, 2006.

\bibitem{gowers2}
W.~T. Gowers.
\newblock Hypergraph regularity and the multidimensional {S}zemer\'edi theorem.
\newblock {\em Ann. of Math. (2)}, 166(3):897--946, 2007.

\bibitem{LinWolf}
Y.~Lin and J.~Wolf.
\newblock On subsets of {$\Bbb F^n_q$} containing no {$k$}-term progressions.
\newblock {\em European J. Combin.}, 31(5):1398--1403, 2010.

\bibitem{rodl1}
B.~Nagle and V.~R\"odl.
\newblock Regularity properties for triple systems.
\newblock {\em Random Structures Algorithms}, 23(3):264--332, 2003.

\bibitem{rodl2}
B.~Nagle, V.~R\"odl, and M.~Schacht.
\newblock The counting lemma for regular {$k$}-uniform hypergraphs.
\newblock {\em Random Structures Algorithms}, 28(2):113--179, 2006.

\bibitem{nagle-rodl-schacht}
B.~Nagle, V.~R\"{o}dl, and M.~Schacht.
\newblock Extremal hypergraph problems and the regularity method.
\newblock In {\em Topics in discrete mathematics}, volume~26 of {\em Algorithms
  Combin.}, pages 247--278. Springer, Berlin, 2006.

\bibitem{cbc}
M.~B. Paterson, D.~R. Stinson, and R.~Wei.
\newblock Combinatorial batch codes.
\newblock {\em Adv. Math. Commun.}, 3(1):13--27, 2009.

\bibitem{Rodlpacking}
V.~R\"{o}dl.
\newblock On a packing and covering problem.
\newblock {\em European J. Combin.}, 6(1):69--78, 1985.

\bibitem{rodl3}
V.~R\"odl and J.~Skokan.
\newblock Regularity lemma for {$k$}-uniform hypergraphs.
\newblock {\em Random Structures Algorithms}, 25(1):1--42, 2004.

\bibitem{rodl4}
V.~R\"odl and J.~Skokan.
\newblock Counting subgraphs in quasi-random 4-uniform hypergraphs.
\newblock {\em Random Structures Algorithms}, 26(1-2):160--203, 2005.

\bibitem{ruzsa}
I.~Z. Ruzsa.
\newblock Solving a linear equation in a set of integers. {I}.
\newblock {\em Acta Arith.}, 65(3):259--282, 1993.

\bibitem{RS}
I.~Z. Ruzsa and E.~Szemer{\'e}di.
\newblock Triple systems with no six points carrying three triangles.
\newblock In {\em Combinatorics ({P}roc. {F}ifth {H}ungarian {C}olloq.,
  {K}eszthely, 1976), {V}ol. {II}}, volume~18 of {\em Colloq. Math. Soc.
  J\'anos Bolyai}, pages 939--945. North-Holland, Amsterdam-New York, 1978.

\bibitem{other2}
G.~N. S{\'a}rk{\"o}zy and S.~Selkow.
\newblock An extension of the {R}uzsa-{S}zemer\'edi theorem.
\newblock {\em Combinatorica}, 25(1):77--84, 2005.

\bibitem{other1}
G.~N. S{\'a}rk{\"o}zy and S.~Selkow.
\newblock On a {T}ur\'an-type hypergraph problem of {B}rown, {E}rd{\H o}s and
  {T}. {S}\'os.
\newblock {\em Discrete Math.}, 297(1-3):190--195, 2005.

\bibitem{shf}
C.~Shangguan and G.~Ge.
\newblock Separating hash families: a {J}ohnson-type bound and new
  constructions.
\newblock {\em SIAM J. Discrete Math.}, 30(4):2243--2264, 2016.

\bibitem{shangguan2020density}
C.~Shangguan and I.~Tamo.
\newblock Degenerate {T}ur\'{a}n densities of sparse hypergraphs.
\newblock {\em J. Combin. Theory Ser. A}, 173:105228, 25, 2020.

\bibitem{shangguan2019universally}
C.~Shangguan and I.~Tamo.
\newblock Sparse hypergraphs with applications to coding theory.
\newblock {\em SIAM J. Discrete Math.}, 34(3):1493--1504, 2020.

\bibitem{smallcores}
D.~Solymosi and J.~Solymosi.
\newblock Small cores in 3-uniform hypergraphs.
\newblock {\em J. Combin. Theory Ser. B}, 122:897--910, 2017.

\bibitem{Szemeredi}
E.~Szemer\'{e}di.
\newblock Regular partitions of graphs.
\newblock In {\em Probl\`emes combinatoires et th\'{e}orie des graphes
  ({C}olloq. {I}nternat. {CNRS}, {U}niv. {O}rsay, {O}rsay, 1976)}, volume 260
  of {\em Colloq. Internat. CNRS}, pages 399--401. CNRS, Paris, 1978.

\bibitem{turan}
P.~Tur\'{a}n.
\newblock Eine {E}xtremalaufgabe aus der {G}raphentheorie.
\newblock {\em Mat. Fiz. Lapok}, 48:436--452, 1941.

\end{thebibliography}

\end{document}